\documentclass[11pt]{article}

\usepackage[hmarginratio=1:1,vmarginratio=1:1,textwidth=4.6in,textheight=7.5in]{geometry}

\usepackage{bbm}
\usepackage[intlimits,sumlimits]{amsmath}
\usepackage{amsthm}
\usepackage{amssymb}
\usepackage{amscd}
\usepackage[all,cmtip]{xy}

\newcommand{\End}{\mathrm{End}}

\newcommand{\id}{\mathrm{id}}

\newcommand{\Hom}{\mathrm{Hom}}
\newcommand{\ihom}{\underline{\mathrm{Hom}}}
\newcommand{\sdiff}{\mathrm{SDiff}}

\newcommand{\Aut}{\mathrm{Aut}}
\newcommand{\Sym}{\mathrm{Sym}}
\newcommand{\pder}[2]{\frac{\partial{#1}}{\partial{#2}}}
\newcommand{\gh}{\hat{\Gamma}}
\newcommand{\scinfty}{SC^\infty}

\newcommand{\intZ}{\mathbbm{Z}}
\newcommand{\realR}{\mathbbm{R}}

\newcommand{\olr}{\overline{\realR}}
\newcommand{\olv}{\overline{V}}

\newcommand{\catgr}{\mathsf{Gr}}
\newcommand{\catsets}{\mathsf{Sets}}
\newcommand{\catvect}{\mathsf{Vect}}
\newcommand{\catsvect}{\mathsf{SVect}}

\newcommand{\catsalg}{\mathsf{SAlg}}

\newcommand{\cattop}{\mathsf{Top}}
\newcommand{\catman}{\mathsf{Man}}
\newcommand{\catsman}{\mathsf{SMan}}
\newcommand{\catvbun}{\mathsf{VBun}}
\newcommand{\catsvbun}{\mathsf{SVBun}}
\newcommand{\catc}{\mathsf{C}}
\newcommand{\catd}{\mathsf{D}}
\newcommand{\catspoint}{\mathsf{SPoint}}
\newcommand{\catgrp}{\mathsf{Grp}}

\newcommand{\cp}{\mathcal{P}}
\newcommand{\cm}{\mathcal{M}}
\newcommand{\cg}{\mathcal{G}}
\newcommand{\cu}{\mathcal{U}}
\newcommand{\cn}{\mathcal{N}}
\newcommand{\cx}{\mathcal{X}}
\newcommand{\cj}{\mathcal{J}}
\newcommand{\cv}{\mathcal{V}}
\newcommand{\ce}{\mathcal{E}}
\newcommand{\ca}{\mathcal{A}}
\newcommand{\cd}{\mathcal{D}}
\newcommand{\ct}{\mathcal{T}}
\newcommand{\ctm}{\mathcal{TM}}

\newcommand{\evf}{\cx(\cm)_{\bar{0}}}

\newtheorem{thm}{Theorem}[section]
\newtheorem{prop}[thm]{Proposition}
\newtheorem{lemma}[thm]{Lemma}
\newtheorem{cor}[thm]{Corollary}
\newtheorem{dfn}[thm]{Definition}

\newtheorem{ex}[thm]{Example}

\newcommand{\ovf}{\cx(\cm)_{\bar{1}}}
\newcommand{\Der}{\ensuremath{\operatorname{Der}}}
\newcommand{\icomp}{\underline{\mathrm{\circ}}}
\newcommand{\catmod}{\mathsf{Mod}}

\usepackage{amssymb,amsmath,amscd,dsfont}
\usepackage[notref,notcite,final]{showkeys}
\usepackage[pdftex]{hyperref}

\newcommand{\op}[1]{\ensuremath{\operatorname{#1}}}

\newcommand{\wh}[1]{\ensuremath{\widehat{#1}}}
\newcommand{\ol}[1]{\ensuremath{\overline{#1}}}
\newcommand{\ul}[1]{\ensuremath{\underline{#1}}}

\newcommand{\cG}{\ensuremath{\mathcal{G}}}

\newcommand{\cV}{\ensuremath{\mathcal{V}}}

\newcommand{\fg}{\ensuremath{\mathfrak{g}}}

\newcommand{\mf}[1]{\ensuremath{\mathfrak{#1}}}

\newcommand{\C}{\ensuremath{\mathds{C}}}

\newcommand{\R}{\ensuremath{\mathds{R}}}

\newcommand{\N}{\ensuremath{\mathds{N}}}
\newcommand{\Z}{\ensuremath{\mathds{Z}}}

\newcommand{\gau}{\ensuremath{\operatorname{\mathfrak{gau}}}}

\newcommand{\aut}{\ensuremath{\operatorname{\mathfrak{aut}}}}

\newcommand{\Diff}{\ensuremath{\operatorname{Diff}}}

\newcommand{\Gau}{\ensuremath{\operatorname{Gau}}}

\newcommand{\im}{\ensuremath{\operatorname{im}}}

\newcommand{\tx}[1]{\ensuremath{\text{#1}}}

\newcommand{\GL}{\ensuremath{\operatorname{GL}}}

\newcommand{\se}{\ensuremath{\nobreak\subseteq\nobreak}}
\newcommand{\from}{\ensuremath{\nobreak:\nobreak}}
\renewcommand{\to}{\ensuremath{\nobreak\rightarrow\nobreak}}

\newcommand{\catpt}{\mathsf{Pt}}
\newcommand{\catspt}{\mathsf{SPt}}

\let\phi=\varphi

\begin{document}

\title{The diffeomorphism supergroup of a finite-dimensional supermanifold}
\author{Christoph Sachse$^a$, Christoph Wockel$^b$\\
\\
$^a$ Dept. of Mathematics, The University of Texas at Austin\\
1 University Station C1200, 78712 Austin, TX, USA \\
\texttt{chsachse@googlemail.com}\\[2em]
$^b$ Mathematisches Institut\\
Georg-August-Universit\"at G\"ottingen\\ 
Bunsenstra\ss e 3-5,
D-37073 G\"ottingen, Germany\\
\texttt{christoph@wockel.eu}
}
\date{}
\maketitle

\begin{abstract}
Using the categorical description of supergeometry we give an explicit construction of the diffeomorphism
supergroup of a compact finite-dimensional supermanifold. The construction provides the diffeomorphism supergroup
with the structure of a Fr\'echet supermanifold. In addition, we derive results about the structure of
diffeomorphism supergroups.
\end{abstract}

\section{Introduction}
Groups of smooth diffeomorphisms are of great importance for numerous applications in geometry, global analysis
and mathematical physics. To give these groups the structure of a Lie group is, however, often a quite non-trivial
task due to the fact that in general one can only endow spaces of smooth maps with a Fr\'echet structure. In almost all
cases of interest, Banach structures are unavailable (cf.\cite[Cor.\
IX.1.7]{Neeb06Towards-a-Lie-theory-of-locally-convex-groups} and
\cite{Omori78On-Banach-Lie-groups-acting-on-finite-dimensional-%
manifolds}).
This makes for an analytically much more challenging situation.

While these difficulties have been overcome for ordinary smooth
manifolds decades ago (cf.\ \cite{Neeb06Towards-a-Lie%
-theory-of-locally-convex-groups} and references therein), no similar
results are available yet for supermanifolds because a theory of
infinite-dimensional supermanifolds has never been systematically
developed. The foundation for such a theory has been laid by Molotkov
already in 1984 \cite{Molotkov94Infinite-dimensnional-Z2k-super-manifolds} but was not really
appreciated at that time. We will follow this line of thought, building
on the results of \cite{Sachse08A-Categorical-%
Formulation-of-Superalgebra-and-Supergeometry}, which works out a
categorical description of supergeometry in detail. This description
makes Banach- and Fr\'echet supermanifolds available, among other
things.

In this article we show that the supergroup of diffeomorphisms of a compact finite-dimensional
supermanifold can be given the structure of a Fr\'echet supermanifold, using the formalism of \cite{Sachse08A-Categorical-%
Formulation-of-Superalgebra-and-Supergeometry}.
To arrive at this assertion, we establish a structure theorem for diffeomorphism supergroups which shows that
superdiffeomorphisms can be factorized in a particular way which allows to decompose the supergroup into
a sequence of semidirect products. This enables us to treat the underlying group separately. Here is where the
main analytic difficulties have to be overcome. The remaining part of the supergroup (the ``higher points'') is
then easier to deal with.

\section{Categorical description of supermanifolds}

We will only give a very condensed review of the categorical description
of supermanifolds. For more details see \cite{Sachse08A-Categorical-%
Formulation-of-Superalgebra-and-Supergeometry} and
\cite{Molotkov%
94Infinite-dimensnional-Z2k-super-manifolds}. 

The main idea of this approach is to
first set up a proper notion of a superset (as a functor) and then to develop all more
advanced concepts from this basic notion. Recall that an ordinary set
$X$ can be described as $\Hom_{\catsets}(\{*\},X)$ where $\{*\}$ is a one-point set.
Even more trivially, $X$ can be viewed as a functor $\catpt\to \catsets$ (where $\catpt$ is a category
with one element and its identity morphism) and a map is a natural
transformation between two such functors.

From this point of view, a
superset will be a functor from a category $\catspt$ of "superpoints" to
$\catsets$. Consequently, a supermanifold will be defined to be a
superset, which is locally isomorphic to certain subfunctors of
$\catspt\to \catvect$. The great advantage of this rather abstract
formalism is that it can treat infinite-dimensional supermanifolds on
the same footing as finite-dimensional ones, in contrast with the usual
ringed-space approach. 

\subsection{The Category of supermanioflds}

Throughout this article, the terms ``super vector space''
and $\intZ_2$ graded vector space are used synonymously. On the level of
vector spaces (or, more generally, modules over superrings) these two notions are identical.
The difference lies in the braiding of these categories, i.e., in the notion of
supercommutativity.

\begin{dfn}
 The category $\catgr$ of finite-dimesnional Grassmann algebras has for
 each $n\in \N_{0}$ an object $\Lambda_{n}$, which is the the (isomorphism class of any) free
 supercommutative algebra on $n$ odd generators.\footnote{i.e.,
 $\Lambda_{n}\cong \Lambda^{\bullet}(\R^{n})=\Lambda^{\op{even}}(\R^{n})\oplus \Lambda^{\op{odd}}(\R^{n})$,
 which is $\Z_{2}$-graded and satisfies
 $v\wedge w=(-1)^{|v|\cdot |w|}w\wedge v$.} Morphisms in $\catgr$ are
 morphisms of $\Z_{2}$-graded algebras. The category $\catspt$ of
 finite-dimensional super points has objects
 $\cp(\Lambda_{n}):=(\{*\},\Lambda_{n})$, i.e. the one-point space $\{*\}$ endowed with the structure sheaf
 $\Lambda_{n}$ and morphisms
 $(\id,\varphi^*)\from (\{*\},\Lambda_{m})\to(\{*\},\Lambda_{n})$ for
 $\varphi\from \Lambda_{n}\to\Lambda_{m}$ a morphism in $\catgr$.
\end{dfn}

Obviously, $\catspt$ is dual to $\catgr$ and thus
$\catsets^{\catspt^\circ}\cong\catsets^{\catgr}$.\footnote{The category of covariant functors $\catc\to\catd$ will be denoted as $\catd^\catc$.} With this said, the basic idea of "superification"
is quite clear, one has to rephrase each classical concept in terms of
the functor category $\catsets^{\catgr}$. The way how to achieve this can be subtle, though, because we
have to make sure the resulting functors really describe the known super objects like, e.g., super vector
spaces. Just like not all functors $\catc^\circ\to\catsets$ describe objects of $\catc$, i.e., are representable,
not all functors in $\catsets^\catgr$ of some given type will represent a super object. For example, not all
functors $\catgr\to\catvect$ actually describe super vector spaces. Below we will briefly state which such
functors are \emph{superrepresentable}. For more details, the reader is referred to \cite{Sachse08A-Categorical-%
Formulation-of-Superalgebra-and-Supergeometry} and
\cite{Molotkov%
94Infinite-dimensnional-Z2k-super-manifolds}.

As a starting point one rephrases superalgebra as algebra in the functor category $\catsets^\catgr$. To each
super vector space $V$ one associates a functor $\olv\in\catsets^\catgr$ as follows:

\begin{ex}\label{ex:superrepresentableOlrModlue}
 For each $\Z_{2}$-graded (= super) vector space $V$ we obtain a functor
 $\ol{V}\from \catgr\to\catsets$, defined by
 \begin{align*}
  \olv:\Lambda_{n} & \mapsto  (\Lambda_{n}\otimes V)_{\bar{0}},\\
  \varphi:\Lambda_{n}\to\Lambda_{m} &\mapsto 
  \varphi\otimes\id_V\big|_{\olv(\Lambda_{n})}.
 \end{align*}
 This is a module over the superring $\olr$, obtained from plugging
 $\realR$ into the above definition. Moreover, if $f\from V_1\times\ldots\times V_n\to V$ is
 a multilinear parity preserving map between super vector spaces, then we define a natural
 transformation
 \begin{eqnarray}
 \label{fbar}
\overline{f}:\olv_1\times\ldots\olv_n &\to & \olv\\
\nonumber
\overline{f}_\Lambda(\lambda_1\otimes v_1,\ldots,\lambda_n\otimes v_n) &\mapsto & \lambda_n\cdots\lambda_1 \otimes f(v_1,\ldots,v_n)
 \end{eqnarray}
 This results in a functor
 $\ol{\cdot}\from \catsvect\to\catmod_{\olr}\subset\catsets^\catgr$.
\end{ex}

The functor $\overline{\cdot}$ can be shown to be fully faithful
\cite[Cor.\ 3.2]{Sachse08A-Categorical-%
Formulation-of-Superalgebra-and-Supergeometry}. An object $\cv\in\catsets^\catgr$ in
the essential image of $\overline{\cdot}$ is called a superrepresentable
$\olr$-module. These superrepresentable $\olr$-modules play the role in
super-differential geometry that vector spaces play in ordinary
differential geometry.

(Smooth) supermanifolds are now defined as functors $\catgr\to\catsets$
which are locally modeled on superrepresentable $\olr$-modules. Note
that if we restrict $\ol{\cdot}$ to the category of locally convex
vector spaces and continuous linear maps, then we can endow each of the
vector spaces $\olv(\Lambda_{n})$ in the image of a functor $\olv$ with
a topology, because its definition only involves tensor products with
the finite-dimensional vector spaces $\Lambda_{n}$. Moreover all induced
maps $\olv(\varphi)$ (for $\varphi\from \Lambda_{n}\to \Lambda_{m}$ a
morphism in $\catgr$) become continuous, and $\olv$ is actually an
object of $\cattop^\catgr$. The category $\cattop^\catgr$ can be given a
Grothendieck topology by pulling back the global classical topology on
$\cattop$ \cite{Sachse08A-Categorical-%
Formulation-of-Superalgebra-and-Supergeometry}. In the following we will assume
$\cattop^\catgr$ and all its relevant subcategories to be endowed with
this topology. The topology on $\cattop^\catgr$ in particular provides
the notion of an open subfunctor of a superrepresentable $\olr$-module.
Note that the treatment of infinite-dimensional super manifolds
is tacitly covered by this approach.

We will be particularly interested in the case where $V$ has been endowed
with the structure of a Fr\'echet space. Functors which are isomorphic
to open subfunctors of such superrepresentable Fr\'echet $\olr$-modules
will be called Fr\'echet superdomains. There is a natural notion of
supersmooth morphisms between such superdomains \cite[Sect.\
4.2]{Sachse08A-Categorical-%
Formulation-of-Superalgebra-and-Supergeometry}, allowing for the following definition (cf.\
\cite[Sect.\ 4.4]{Sachse08A-Categorical-%
Formulation-of-Superalgebra-and-Supergeometry}).

\begin{dfn}
 Denoting by $\catman$ the category of smooth Fr\'echet manifolds, a
 supermanifold $\cm$ is a functor $\catgr\to \catman$ endowed with a maximal atlas.
 An atlas consists of
 \begin{itemize}
  \item an open cover
        $\{\cu_\alpha\to\cm\}_{\alpha\in A}$ by Fr\'echet superdomains
        such that
  \item each pullback $\cu_{\alpha\beta}=\cu_\alpha\times_\cm\cu_\beta$
        is a superdomain and
  \item the canonical projections
        $\Pi_{\alpha,\beta}:\cu_{\alpha\beta}\to\cu_\alpha,\cu_\beta$
        are supersmooth.
 \end{itemize}
 A morphism $\phi:\cm\to\cm'$ of supermanifolds is a natural transformation in $\catman^\catgr$ such that
 for every chart $u:\cu\to\cm$ and $u':\cu'\to\cm'$ the diagram
 \[
 \xymatrix{
  \cu\times_{\cm'}\cu' \ar[d]_{\pi}\ar[rr]^{\pi'} && \cu' \ar[d]^{u'}\\
  \cu \ar[r]^{u} & \cm \ar[r]^{\phi} & \cm'
 }\]
 commutes. As usual, two atlases are equivalent if their union is again an atlas. This entails the notion of
 a maximal atlas.
\end{dfn}

Together with the corresponding supersmooth morphisms, we will denote by $\catsman$ the category of
Fr\'echet supermanifolds.

\subsection{Inner Hom objects in $\catsman$}
\label{sect:smanihom}

The subcategory $\catspoint\subset\catsman$ of super points plays a special role for the category of supermanifolds,
analogous to that played by the one-point manifold for the category of ordinary manifolds. This is best seen from
the fact
\cite{Sachse08A-Categorical-Formulation-of-Superalgebra-and-Supergeometry} that
\[
\Hom(\cp(\Lambda),\cm)\cong\cm(\Lambda)
\]
for all $\Lambda\in\catgr$ and any supermanifold $\cm$. Moreover, this isomorphism is functorial in $\Lambda$ as well
as in $\cm$. So the $\Lambda$-points (i.e., the sets $\cm(\Lambda)$) of $\cm$ are indeed given by all the 
possible maps of $\cp(\Lambda)$ into $\cm$.

An important consequence for our purpose is that this gives a hint on how to describe inner Hom objects in
$\catsman$. An inner Hom object $\ihom(B,C)$ in any category $\catc$ is required to satisfy the adjunction formula \cite{Mac-Lane98Categories-for-the-working-mathematician}
\[
\Hom(A,\ihom(B,C))\cong\Hom(A\times B,C)\qquad\forall\,\,A,B,C\in\catc.
\]
Therefore, given two supermanifolds $\cm,\cn$ the $\Lambda$-points of $\ihom(\cm,\cn)$ are given by
\[
\ihom(\cm,\cn)(\Lambda)\cong\Hom_\catsman(\cp(\Lambda)\times\cm,\cn).
\]
This is as stated only a relation between sets.
The hard part is, of course, to give these sets manifold structures such that $\ihom(\cm,\cn)$ becomes a
supermanifold. If $\cm,\cn$ are not discrete then this will, at best, be possible within the category of
Fr\'echet supermanifolds. 

In general, the study of such inner Hom objects is an analytically very challenging problem already for ordinary
manifolds. We will only attempt to make this notion precise in two cases in this paper: we will define and study
the \emph{space} of sections of a super vector bundle over a supermanifold. As one may expect, it will turn out to
be a superrepresentable $\olr$-module. Although this is of course expected it is not obvious, in contrast to ordinary geometry, because even the notion of a section over a space which is not described by its underlying topological points is a bit involved.
The second example and overall goal will be the explicit construction of the diffeomorphism supergroup $\sdiff(\cm)$ of a compact supermanifold studied below. This supergroup will turn out to be a subobject of $\ihom(\cm,\cm)$ in a
way that we will make precise. 

\section{Supergroups}
\label{sect:sgrp}

The most well-known example of a supergroup is the following:

\begin{dfn}
\label{def:sgroup}
A Lie supergroup is a group object in the category of supermanifolds.
\end{dfn}

More explicitly, a supermanifold $\cg$ is turned into a supergroup by specifying morphisms
\begin{eqnarray*}
m:\cg\times\cg &\to &\cg\\
i:\cg &\to &\cg\\
e:\realR^0=\{*\} &\to &\cg
\end{eqnarray*}
which satisfy a number of diagrams encoding the axioms of a group \cite{Mac-Lane98Categories-for-the-working-mathematician}. For example, associativity amounts in this
language to the condition
\[
m\circ(m\times \id_\cg)=m\circ(\id_\cg\times m).
\]
Instead of requiring the commutativity of certain diagrams one can equivalently require that the set
of $T$-points $\cg(T)=\Hom(T,\cg)$ is a group for every supermanifold $T$ and that this family of groups is
natural in $T$, i.e., that multiplication, inversion and unit are given by the induced maps 
$m_T:\cg(T)\times\cg(T)\to\cg(T)$ and $i_T, e_T$, respectively.

That definition \ref{def:sgroup} only deals with \emph{Lie} supergroups reflects the fact that at first it seems unclear how to generalize the concept of a group as a set with a certain structure to something ``super''. 
One way to escape this limitation is to give up thinking of structured sets, as indeed suggested by Def. \ref{def:sgroup}.
In view of the categorical formulation sketched in the previous section, we should rather think of a family
of sets related by functoriality in $\catgr$:

\begin{dfn}
A supergroup is a group object in $\catsets^\catgr$.
\end{dfn}

This obviously includes Lie supergroups as defined above, but also more general objects. As a subcategory we
obtain, for example, ``topological supergroups'', which we define as groups in $\cattop^\catgr$.
The study of these more general supergroups should be interesting in its own right. In addition, the orbits
and orbit spaces of supergroup actions on supermanifolds often turn out not to be supermanifolds. However,
they are always objects in $\catsets^\catgr$ which suggests this topos as the natural ``habitat'' to study
supergroups.
In this work, however, we will restrict ourselves to supergroups which can be endowed with the structure of
a supermanifold.

Let $\cg$ be a group object in $\catsets^\catgr$. Then every $\cg(\Lambda)$ is a group, i.e., $\cg$ is actually
a functor $\catgr\to\catgrp$. The initial and terminal morphisms $c_\Lambda:\realR\to\Lambda$ and
$\epsilon_\Lambda:\Lambda\to\realR$ induce homomorphisms
\[
\cg(c_\Lambda):\cg(\realR)\to\cg(\Lambda),\qquad \cg(\epsilon_\Lambda):\cg(\Lambda)\to\cg(\realR).
\]
Since $\epsilon_\Lambda\circ c_\Lambda=\id_\realR$, $\cg(c_\Lambda)$ is a monomorphism and $\cg(\epsilon_\Lambda)$ is an
epimorphism. This means that for every $\Lambda\in\catgr$ we can write
\begin{equation}
\label{eq:e1}
\cg(\Lambda)= \cn(\Lambda)\rtimes G
\end{equation}
where $G:=\cg(\realR)\cong \im(\cg(c_{\Lambda}))$ and $\cn(\Lambda):=\ker(\cg(\epsilon_{\Lambda}))$.

We can even say more. For every morphism $\phi:\Lambda\to\Lambda'$ in $\catgr$ we have that $\epsilon_{\Lambda'}\circ\phi=\epsilon_\Lambda$. Thus
\[
\xymatrix{
\cg(\Lambda) \ar[d]_{\cg(\epsilon_\Lambda)}\ar[r]^{\cg(\phi)}& \cg(\Lambda')\ar[d]^{\cg(\epsilon_{\Lambda'})}\\
G \ar[r]^{\id_G}& G
}
\]
commutes. Therefore \eqref{eq:e1} can be read as a component equation for the splitting
\begin{equation}
\label{eq:e2}
\cg=\cn\rtimes G
\end{equation}
where $G$ is interpreted as the constant functor $\catgr\to\catgrp$ with value $G$ which sends each morphism to
$\id_G$ and $\cn$ is the supergroup $\Lambda\mapsto \cn(\Lambda)$ and $\phi\mapsto \cg(\phi)\big|_{\cn(\Lambda)}$
for all morphisms $\phi$ in $\catgr$.

Let us now assume $\cg$ is a Lie supergroup. This implies that all $\cg(\Lambda)$ are Lie groups which moreover
have a rather special structure. We again have the maps $\cg(\epsilon_\Lambda)$, $\cg(c_\Lambda)$ with their
respective properties. The Lie supergroup $\cg$ is locally modeled on a linear superspace which we may identify with
its Lie superalgebra $\fg=\fg_{\bar{0}}\oplus\fg_{\bar{1}}$. In particular, there has to exist a superchart $\phi\from \cu\to \cG$ around 
the identity. The underlying chart $\phi_{\realR}$ is a chart around $1$ for $G$ which we
may identify with a map $\phi_\realR:\fg_{\bar{0}}\supset U\to G$. This map might not be the exponential map if we
are in the infinite-dimensional context.

The existence of a superchart means that we can extend $\phi_\realR$ for each $\Lambda$ to a chart
$\phi_\Lambda:\cu(\Lambda)\to\cg(\Lambda)$ where $\cu$ is an open superdomain in $\fg$. The fibers of the map
\[
\cg(\epsilon_\Lambda):\cg(\Lambda)\to\cg(\realR)
\]
are therefore linear spaces isomorphic to
\[
(\fg_{\bar{0}}\otimes\Lambda_{\bar{0}}^{nil})\oplus(\fg_{\bar{1}}\otimes\Lambda_{\bar{1}})
\]
where $\Lambda_{\bar{0}}^{nil}$ denotes the nilpotent ideal in $\Lambda_{\bar{0}}$.

These linear spaces do not form a superrepresentable $\olr$-module \cite{Sachse08A-Categorical-Formulation-of-Superalgebra-and-Supergeometry}, which means that one cannot model a supermanifold on them. Similarly, a constant functor $\catgr\to\catman$ cannot be a supermanifold.
Consequently the direct sum splitting \eqref{eq:e2} cannot exist in the category of Lie supergroups.
Nonetheless it turns out to be very useful in the construction of supercharts.
In our discussion of the supergroup of diffeomorphisms of a supermanifold below we will exhibit the
splitting \eqref{eq:e1} explicitly.

\section{Super vector bundles}
\label{sect:superVectorBundles}

\subsection{Definition}

In this section we will present a brief but hopefully self-contained treatment of super vector bundles in
the categorical approach.

The construction of super vector bundles is formally completely analogous to that of ordinary vector bundles.
The definition we will present was first given in \cite{Molotkov%
94Infinite-dimensnional-Z2k-super-manifolds}. A trivial smooth super vector bundle
is given by $\pi_\cm:\cm\times\cv\to\cm$, where $\cm$ is a smooth supermanifold, $\pi_\cm$ is the canonical projection and $\cv$ is a linear supermanifold, i.e., a topological 
superrepresentable $\olr$-module.
Morphisms are pairs $(f:\cm\to\cm',g:\cm\times\cv\to\cm'\times\cv')$ such that 
\[
\pi_{\cm'}\circ g=f\circ\pi_\cm
\]
and such that $\pi_{\cv'}\circ g:\cm\times\cv\to\cv'$ is a 
$\cm$-family \cite{Molotkov94Infinite-dimensnional-Z2k-super-manifolds}, \cite{ DeligneMorgan99Notes-on-supersymmetry} of isomorphisms of $\olr$-modules.
The latter condition is the categorified version of being a fiberwise isomorphism. The term ``fiber'' must
be used with caution when speaking about super vector bundles because the base manifold is not described as
a collection of ordinary topological points.
Thus trivial super vector bundles are certain functors $\catgr\to\catvbun$, where $\catvbun$ are smooth super vector
bundles over a smooth base. 

Note that every functor $\ce\in\catvbun^\catgr$ gives rise to a functor $\cm\in\catman^\catgr$
by assigning to every component bundle its base manifold.

\begin{dfn}
Let $\ce,\ce'$ be functors in $\catvbun^\catgr$, and let $\cm,\cm'$ be their associated base functors in
$\catman^\catgr$. Then $\ce$ is said to be an open subfunctor of $\ce'$, denoted $\ce\subset\ce'$, if
\begin{enumerate}
\item $\cm$ is an open subfunctor of $\cm'$, and
\item for each $\Lambda\in\catgr$ we have $\pi_\Lambda^{-1}(\cm(\Lambda))={\pi'}_{\Lambda}^{-1}(\cm(\Lambda))$,
\end{enumerate}
where $\pi_\Lambda:\ce(\Lambda)\to\cm(\Lambda)$ is the projection to the base.
\end{dfn}

A morphism $\ce''\to\ce$ of functors in $\catvbun^\catgr$
is called open if it can be factorized as a composition
\begin{equation*}
\begin{CD}
\ce'' @>f>> \ce' \subset\ce
\end{CD},
\end{equation*}
where $f$ is an isomorphism of functors and $\ce'$ is an open subfunctor of $\ce$. An open 
covering $\{\ce_\alpha\}_{\alpha\in A}$ of $\ce\in\catvbun^\catgr$
is then a collection of open morphisms $\{\phi_\alpha:\ce_\alpha\to\ce\}_{\alpha\in A}$, such 
that the associated maps
$\{\pi\circ\phi_\alpha\}_{\alpha\in A}$ are an open covering of the functor $\cm:\catgr\to\catman$ 
associated with 
$\ce$. In analogy with supermanifolds, a supervector bundle is a functor in $\catvbun^\catgr$ endowed with an
atlas of trivial open subbundles.

\begin{dfn}
\label{def:svbun}
Let $\ce$ be a functor in $\catvbun^\catgr$, and let $\cm\in\catman^\catgr$ be its associated functor
of base manifolds. Let $\ca=\{\phi_\alpha:\ce_\alpha\to\ce\}_{\alpha\in A}$ be an open covering of $\ce$. 
Then this covering is an atlas of a
super vector bundle $\ce$ over the supermanifold $\cm$ if the following conditions hold:
\begin{enumerate}
\item each of the $\ce_\alpha$ is a trivial super vector bundle $\cu_\alpha\times\cv_\alpha$, and
$\cv_\alpha\cong\cv_\beta$ for all $\alpha,\beta\in A$, and
\item for each $\alpha,\beta\in A$, the overlaps
\begin{equation*}
\xymatrix{
\ce_\alpha\times_\ce\ce_\beta \ar[r]^(.6){\pi_\alpha}\ar[d]_{\pi_\beta}& \ce_\alpha\ar[d]^{\phi_\alpha}\\
\ce_\beta \ar[r]^{\phi_\beta}& \ce
}
\end{equation*}
can be given the structure of a trivial super vector bundle in such a way that the projections
$\pi_\alpha,\pi_\beta$ become morphisms of trivial super vector bundles.
\end{enumerate}
Two atlases $\ca$ and $\ca'$ are equivalent, if their union $\ca\cup\ca'$ is again an atlas. A super
vector bundle $\ce$ is a functor in $\catvbun^\catgr$ together with an equivalence class of atlases.
\end{dfn}

The second condition is necessary because the fiber product in the diagram is constructed as the fiber
product in $\catvbun^\catgr$. We thus have to make sure that it actually exists in the subcategory of
trivial super vector bundles. Note also that
the requirement that the transition functions be morphisms of trivial super vector bundles automatically
turns $\cm$ into a supermanifold.

\begin{dfn}
\label{def:svbunmor}
Let $\ce,\ce'$ be super vector bundles with open coverings $\{\phi_\alpha:\ce_\alpha\to\ce\}_{\alpha\in A}$ 
and
$\{\phi_{\alpha'}:\ce'_{\alpha'}\to\ce'\}_{\alpha'\in A'}$. A functor morphism $\Phi:\ce\to\ce'$ in 
$\catvbun^\catgr$ is a
morphism of super vector bundles if for all $\alpha\in A$ and all $\alpha'\in A'$, the pullbacks
\begin{equation*}
\xymatrix{
\ce_\alpha\times_{\ce'}\ce_{\alpha'} \ar[rr]^{\pi_{\alpha'}} \ar[d]_{\pi_{\alpha}} &&%
\cu_{\alpha'} \ar[d]^{\phi_{\alpha'}}\\
\cu_\alpha \ar[r]^{\phi_\alpha} & \ce \ar[r]^{\Phi} & \ce'
}
\end{equation*}
can be chosen such that $\ce_\alpha\times_{\ce'}\ce_{\alpha'}$ is a trivial super vector bundle and
the projections $\pi_\alpha,\pi_{\alpha'}$ are morphisms of trivial super vector bundles.
\end{dfn}

Definitions \ref{def:svbun} and \ref{def:svbunmor} yield a category $\catsvbun$ which is obviously
a subcategory of $\catvbun^\catgr$ but not a full one (for basically the same reason for which
$\catman^\catgr$ is not a full subcategory of $\catsman$, cf. \cite{Sachse08A-Categorical-Formulation-of-Superalgebra-and-Supergeometry}).
 One can define super vector bundles in terms
of cocycles with values in a Lie supergroup as well \cite{Molotkov%
94Infinite-dimensnional-Z2k-super-manifolds} but we will not attempt to do this here.

\begin{prop}
\label{trivsvbun}
A super vector bundle $\pi:\ce\to\cm$ is trivial if and only if all of its $\Lambda$-points
$\pi_\Lambda:\ce(\Lambda)\to\cm(\Lambda)$ are trivial bundles.
\end{prop}
\begin{proof}
The bundle $\pi:\ce\to\cm$ is trivial if and only if there exists an isomorphism $f:\ce\to\cm\times\cv$
for some superrepresentable $\olr$-module $\cv$ such that $\pi=\pi_\cm\circ f$. This means that for 
every $\Lambda\in\catgr$, the components of $f$ must make the diagram
\begin{equation*}
\xymatrix{f_\Lambda:\ce(\Lambda)\ar[rr] \ar[dr]_{\pi_\Lambda} & &
\cm(\Lambda)\times\cv(\Lambda) \ar[dl]^{\pi_{\cm,\Lambda}}\\
&\cm(\Lambda)&}
\end{equation*}
commutative. That is precisely the condition for the triviality of the ordinary vector bundle 
$\pi_\Lambda:\ce(\Lambda)\to\cm(\Lambda)$.
\end{proof}

\subsection{The tangent bundle $\boldsymbol{\ctm}$}

The tangent bundle $\ctm$ of a supermanifold $\cm$ is defined in the categorical framework as a functor
$\ctm:\catgr\to\catvbun$ in the following way: for every $\Lambda\in\catgr$ and every 
$\varphi:\Lambda\to\Lambda'$, set
\begin{eqnarray}
\label{eqn:tdef1}
\ctm(\Lambda) &:=& T(\cm(\Lambda)),\\
\nonumber
\ctm(\varphi) &:=& D(\cm(\varphi)):T(\cm(\Lambda))\to T(\cm(\Lambda')).
\end{eqnarray}
To every morphism $f:\cm\to\cm'$ of supermanifolds, we assign a functor morphism
\begin{eqnarray}
\label{eqn:tdef2}
\cd f:\ctm &\to& \ctm'\\
\nonumber
(\cd f)_\Lambda &:=& Df_\Lambda:T(\cm(\Lambda))\to T(\cm'(\Lambda)).
\end{eqnarray}
The assignments \eqref{eqn:tdef1} and \eqref{eqn:tdef2} define the tangent functor $\ct:\catsman\to\catvbun^\catgr$. 
For our definition of a super vector bundle to make sense, we would
certainly expect the tangent bundle to be in $\catsvbun$, not just in $\catvbun^\catgr$. This is indeed
the case:

\begin{prop}
The tangent functor is a functor $\ct:\catsman\to\catsvbun$.
\end{prop}
\begin{proof}
Choose a supersmooth atlas $\{u_\alpha:\cu_\alpha\to\cm\}_{\alpha\in A}$ of $\cm$.
Then all $\cu_\alpha$ are open domains in some superrepresentable
$\olr$-module $\cv$ \cite{Sachse08A-Categorical-%
Formulation-of-Superalgebra-and-Supergeometry}, so their tangent bundles are trivial:
\begin{equation*}
\ct\cu_\alpha\cong\cu_\alpha\times\cv.
\end{equation*}
It is clear that the tangent bundles $\{\ct\cu_\alpha\}_{\alpha\in A}$ of the coordinate domains form
an open cover of the functor $\ct\in\catvbun^\catgr$. It has to be shown that they form an
atlas satisfying the conditions of Definition \ref{def:svbun}.

By the definition of a supermanifold \cite{Sachse08A-Categorical-%
Formulation-of-Superalgebra-and-Supergeometry} each intersection $\cu_\alpha\times_\cm\cu_\beta$ has the structure of a 
superdomain itself, and the projections $\pi_\alpha,\pi_\beta:\cu_{\alpha\beta}\to\cu_\alpha,\cu_\beta$ are supersmooth. The tangent bundles are
related by the differentials, e.g., $\cd\pi_\alpha:\ct\cu_{\alpha\beta}\to\ct\cu_\alpha$. These are by definition
$\cu_{\alpha\beta}$-families of $\olr$-linear morphisms compatible with the base maps. So they are morphisms of
trivial super vector bundles.
\end{proof}

\subsection{Spaces of sections of super vector bundles}

In this Section we present a first application of the categorical
approach to supergeometry. We show that smooth sections of
finite-dimensional super vector bundles form superrepresentable
$\olr$-modules and therefore linear Fr\'echet supermanifolds. This might
seem intuitively clear from ordinary geometry but this intuition is
treacherous in supergeometry. For example, there is no naive notion of
fibers for a super vector bundle and a super vector space is not the
same as a linear supermanifold from the ringed space point of view.
Most of the proofs in this Section rely heavily on results of V. Molotkov \cite{M:private}.

Let $p:\ce\to\cm$ be a smooth super vector bundle over a compact supermanifold $\cm$. We would like
to enrich the set of sections 
\begin{equation*}
\Gamma(\cm,\ce):=\{\sigma\in \Hom_\catsman(\cm,\ce) |\,p\circ\sigma=\id_\cm\}
\end{equation*}
to a supermanifold. We thus have to extend $\Gamma(\cm,\ce)$
to a functor $\hat{\Gamma}:\catgr\to \catsets$ such that its value on
$\realR$ is $\Gamma(\cm,\ce)$. As usual, this can be accomplished by
studying sections of families of super vector bundles over superpoints. 

We define the functor $\gh(\cm,\ce):\catgr\to\catsets$ on the objects of $\catgr$ by setting
\begin{equation*}
\gh(\cm,\ce)(\Lambda):=\Gamma(\cp(\Lambda)\times\cm,\pi_\cm^*\ce).
\end{equation*}
Here, $\pi_\cm^*\ce$
denotes the pullback of $\ce$ along the projection
$\pi_\cm:\cp(\Lambda)\times\cm\to\cm$.
For a morphism $\varphi:\Lambda\to\Lambda'$, we define
\begin{eqnarray}
\label{svbungr}
\gh(\cm,\ce)(\varphi):\gh(\cm,\ce)(\Lambda) &\to& \gh(\cm,\ce)(\Lambda')\\
\nonumber
\sigma &\mapsto& \sigma\circ(\cp(\varphi)\times\id_\cm).
\end{eqnarray}

Note the similarity of this definition to that of inner Hom objects (Section \ref{sect:smanihom}). We do not want to
work out this similarity systematically but only remark that one may use it to introduce the notion of an
inner Hom object in the category of families over a supermanifold $\cm$.
In general, inner Hom objects and even more so functors of the type $\hat{\Gamma}$ for general fiber bundles
are notoriously difficult to endow with additional structure, e.g., supersmooth or superrepresentable $\olr$-module
structures. We will see, however, that this task is feasible here because all fibers are
linear supermanifolds.

Let us first note that the set $\scinfty(\cm,\olv)$ of supersmooth morphisms from a supermanifold $\cm$ into
a superrepresentable $\olr$-module carries a natural vector space structure: if $f,g:\cm\to\olv$ are morphisms then
we define
\[
(f+g)_\Lambda(u):=f_\Lambda(u)+g_\Lambda(u)
\]
and
\[
(r\cdot f)_\Lambda(u):=rf_\Lambda(u)
\]
for $r\in\realR$ and $u\in\cm(\Lambda)$.
If we look at a set of the form $\scinfty(\cm,\olv\oplus\overline{W})$ where $\ol{W}$ is another superrepresentable
$\olr$-module we can even conclude that this set is a $\intZ_2$-graded, i.e., super vector space. The even elements are simply defined to be maps into $\olv$, the odd ones maps into $\ol{W}$.

The following Lemma shows that $\gh(\cu,\cu\times\cv)$ is superrepresentable.

\begin{lemma}\label{lem:superSectionsAreSuperRepresenatble}
Let $\cu$ be a superdomain and $\cu\times\cv\to\cu$ a trivial super vector bundle over $\cu$. Then
\[
\gh(\cu,\cu\times\cv)\cong\overline{\scinfty(\cu,\cv\oplus\Pi\cv)}
\]
as $\olr$-modules.
\end{lemma}
\begin{proof}
We have
\[
\gh(\cu,\cu\times\cv)(\Lambda)\cong\scinfty(\cp(\Lambda)\times\cu,\cv)
\]
and $\cv\cong\olv$ for some super vector space $V$.
On the other hand $\cu=\olv'\big|_U$
for some super vector space $V'$ because $\cu$ was assumed to be a
superdomain\footnote{This means $U\subset V'_{\bar{0}}$ open and $\cu(\Lambda)=\olv'(\epsilon_\Lambda^{-1})(U)$ for all $\Lambda$ in $\catgr$. Cf. \cite{Sachse08A-Categorical-Formulation-of-Superalgebra-and-Supergeo%
metry}}. Since 
\[
\cp(\Lambda_n)(\Lambda)\cong\Hom(\Lambda_n,\Lambda)\cong\Lambda_{\bar{1}}\otimes\realR^n\cong\realR^{0|n} ,
\]
we have
\[
\cp(\Lambda_n)\times\cu=\cp(\Lambda_n)\times\olv'\big|_U\cong\overline{\realR^{0|n}\oplus V'}\big|_U.
\]

As shown in \cite{Sachse08A-Categorical-%
Formulation-of-Superalgebra-and-Supergeometry} the set $\scinfty(\cp(\Lambda)\times\cu,\cv)$ can be identified with the
set of ``skeletons'' of such supersmooth maps. A skeleton of a morphism $f:\cp(\Lambda)\times\cu\to\cv$ consists of
a smooth map $f_0:U\to V_{\bar{0}}$ and a collection of smooth maps $\{f_n:U\to\Sym^n(\realR^{0|n}\oplus V'_{\bar{1}},V)\mid n\geq 1\}$.
Symmetric here of course means a symmetric parity-preserving map of \emph{super} vector spaces, so
\[
\Sym^i((\realR^{0|n}\oplus V')_{\bar{1}},V)=\wedge^i(\realR^n\oplus V'_{\bar{1}},V_{\bar{i}})
\]
where the right hand side denotes alternating maps between \emph{ordinary} vector spaces. Setting 
$\Sym^0(\realR^{0|n}\oplus V'_{\bar{1}},V):=V_{\bar{0}}$ we can identify
\begin{multline*}
\scinfty(\cp(\Lambda_n)\times\cu,\olv)=C^\infty(U,\Sym^\bullet(\realR^{0|n}\oplus V'_{\bar{1}},V))=\\C^\infty(U,\oplus_{i=0}\Sym^i(\realR^{0|n}\oplus V'_{\bar{1}},V)).
\end{multline*}
It is
\begin{eqnarray*}
\Sym^\bullet(\realR^n\oplus V'_{\bar{1}},V) &=& \bigoplus_{i}\wedge^i(\realR^n\oplus V'_{\bar{1}},V_{\bar{i}})\\
&=& \bigoplus_i\bigoplus_{j+k=i}\wedge^j\realR^n\otimes%
\wedge^{k}(V'_{\bar{1}},V_{\bar{i}})\\
&=& \left(\bigoplus_{j\,\,\,\mathrm{even}}\bigoplus_{i=j}^{\infty}\wedge^j\realR^n\oplus\wedge^{i-j}(V'_{\bar{1}},V_{\bar{i}})\right)\oplus\\
&&%
\left(\bigoplus_{j\,\,\,\mathrm{odd}}\bigoplus_{i=j}^{\infty}\wedge^j\realR^n\oplus\wedge^{i-j}(V'_{\bar{1}},V_{\bar{i}})\right)\\
&=& \left(\bigoplus_{j\,\,\,\mathrm{even}}\bigoplus_{m=0}^{\infty}\wedge^j\realR^n\oplus\wedge^{m}(V'_{\bar{1}},V_{\bar{m}})\right)\oplus\\%
&&\left(\bigoplus_{j\,\,\,\mathrm{odd}}\bigoplus_{m=0}^{\infty}\wedge^j\realR^n\oplus\wedge^{m}(V'_{\bar{1}},V_{\overline{m+1}})\right),
\end{eqnarray*}
where on the right hand side, all operations are to be understood as those of ordinary vector spaces.
Now we can rewrite that last line as
\[
\Sym^\bullet(\realR^n\oplus V'_{\bar{1}},V)=\Lambda_{n,\bar{0}}\otimes\Sym^\bullet(V'_{\bar{1}},V)\oplus\Lambda_{n,\bar{1}}\otimes\Sym^\bullet(V'_{\bar{1}},\Pi V)
\]
and therefore
\begin{eqnarray*}
C^\infty(U,\Sym^\bullet(\realR^{0|n}\oplus V'_{\bar{1}},V)) &\cong & \Lambda_{n,\bar{0}}\otimes C^\infty(U,\Sym^\bullet(V'_{\bar{1}},V))%
\oplus\\&&\Lambda_{n,\bar{1}}\otimes C^\infty(U,\Sym^\bullet(V'_{\bar{1}},\Pi V))\\
&\cong & \Lambda_{n,\bar{0}}\otimes \scinfty(\cu,\olv)\oplus\\&&%
\Lambda_{n,\bar{1}}\otimes\scinfty(\cu,\overline{\Pi V})\\
&\cong & (\Lambda_{n}\otimes \scinfty(\cu,\olv\oplus \ol{\Pi V}))_{\ol{0}}\\
&\cong & \ol{\scinfty(\cu,V\oplus \Pi V)}(\Lambda_{n}).
\end{eqnarray*}
\end{proof}

This result stays true if we study a general super vector bundle $\ce\to\cm$ over an arbitrary supermanifold $\cm$.
Note first that the open coverings by trivial bundles as defined above endow the category $\catsvbun$ with a
Grothendieck topology. This topology turns out to be \emph{subcanonical} (for a proof in the very similar case of the category $\catsman$ see \cite{Sachse08A-Categorical-Formulation-of-Superalgebra-and-Supergeometry}). This means that
every representable functor $\catsvbun^\circ\to\catsets$ is a sheaf.

As a consequence, if $\{\phi_\alpha:\ce_\alpha\to\ce\}_{\alpha\in A}$ is an open covering of the super vector bundle $\ce$ then $\ce$ is a colimit with the $\phi_\alpha$ as the canonical maps. More precisely, $\ce$ is the limit of
the diagram
\begin{equation}
\label{eq:diag}
\{\ce_\alpha\longleftarrow\ce_{\alpha\beta}=\ce_\alpha\times_\ce\ce_\beta\longrightarrow\ce_\beta\mid \alpha,\beta\in A\}=:F:\Delta\to\catsvbun
\end{equation}
where $\Delta$ is an abstract diagram category and $F$ a functor into $\catsvbun$ whose image is the open covering
by trivial subbundles and their fibered products.

This in turn entails the following

\begin{lemma}
Taking sections and pull-back maps
\[
\phi_\alpha^*:\hat{\Gamma}(\cm,\ce)\to\hat{\Gamma}(\cu_\alpha,\cu_\alpha\times\cv_\alpha)
\]
we produce a diagram $\hat{\Gamma}(F):\Delta^\circ\to\catmod_{\olr}$ of the sets of sections.
It is
\[
\hat{\Gamma}(\cm,\ce)=\lim(\hat{\Gamma}(F))
\]
as $\olr$-modules.
\end{lemma}
\begin{proof}
In detail, each chart $\phi_\alpha:\ce_\alpha\to\ce$ consists of a pair $(f_\alpha,g_\alpha)$ which makes the
diagram
\[
\xymatrix{\ce_\alpha=\cu_\alpha\times\cv_\alpha \ar[rr]^{g_\alpha} \ar[d] && \ce \ar[d]\\
\cu_\alpha \ar[rr]_{f_\alpha} && \cm}
\]
commute. Here, $g_\alpha$ is a $\cu_\alpha$-family of isomorphisms, so we also have an inverse $g^{-1}_\alpha$.

So given a section $\sigma:\cp(\Lambda)\times\cm\to\ce$ we define the pulled-back section as
\[
\phi_\alpha^*\sigma=g^{-1}_\alpha\circ\sigma\circ(\id_{\cp(\Lambda)}\times f_\alpha)\,\,:\,\,\cp(\Lambda)\times\cu_\alpha\to\ce_\alpha.
\]

Pick some $\Lambda\in\catgr$ and assume we are given local sections $\sigma_\alpha:\cu_\alpha\times\cp(\Lambda)\to\ce_\alpha$ which coincide on the overlaps, i.e.,
\[
\pi_\alpha^*\sigma_\alpha=\pi_\beta^*\sigma_\beta
\] 
where $\pi_\alpha,\pi_\beta$ are the canonical maps of $\ce_\alpha\times_\ce\ce_\beta$.
These local sections define a unique global section $\sigma:\cm\times\cp(\Lambda)\to\ce$ as one immediately checks
pointwise, i.e., by looking at the
\[
\sigma_{\alpha\Lambda'}:\cp(\Lambda)(\Lambda')\times\cu_\alpha(\Lambda')\to\ce_\alpha(\Lambda').
\]
All of these are ordinary (smooth) maps between ordinary spaces which coincide on overlaps.

So for each $\Lambda$, the $\Lambda$-points of local sections of a super vector bundle $\ce$ form a sheaf on
the supermanifold $\cm$. The resulting uniqueness of the patched together global section makes the global sections a limit of the local sections. 
\end{proof}

One can even go one step further and conclude that the \emph{functors} $\hat{\Gamma}(\cu_\alpha,\ce_\alpha)$ form
a sheaf with values in $\olr$-modules on $\cm$.

Since we assume that the $\ce_\alpha=\cu_\alpha\times\cv_\alpha$ are trivial we know from
Lemma \ref{lem:superSectionsAreSuperRepresenatble} that
\[
\hat{\Gamma}(\cu_\alpha,\ce_\alpha)\cong\overline{\scinfty(\cu_\alpha,\cv_\alpha\oplus\Pi\cv_\alpha)}.
\]
Therefore, $\hat{\Gamma}(F):\Delta^\circ\to\catmod_{\olr}$ is a diagram of superrepresentable $\olr$-modules.
If we abbreviate this diagram by an abuse of notation as just $\hat{\Gamma}(\ce_\alpha)$ for a moment then
we note that
\[
\lim(\hat{\Gamma}(\ce_\alpha))\cong\lim(\overline{\scinfty(\cu_\alpha,\cv_\alpha\oplus\Pi\cv_\alpha)})\cong
\overline{\lim(\scinfty(\cu_\alpha,\cv_\alpha\oplus\Pi\cv_\alpha)}
\]
where the last $\cong$ follows from the fact that the functor $\ol{\cdot}$ consists in tensoring with the
finite-dimensional Grassmann algebras and taking the even parts which commutes with limits and colimits.
Thus we have shown

\begin{thm}
Let $\ce\to\cm$ be a real super vector bundle. Then the functor $\hat{\Gamma}(\cm,\ce)$ of global (smooth) sections
is a superrepresentable $\olr$-module.
\end{thm}

From this theorem we can conclude in particular \cite{M:private}
\[
\overline{\lim(\scinfty(\cu_\alpha,\cv_\alpha\oplus\Pi\cv_\alpha)}\cong\ol{\Gamma(\cu_\alpha,\cv_\alpha\oplus\Pi\cv_\alpha)}\cong\ol{\Gamma(\cm,\ce\oplus\Pi\ce)}.
\]
Here, the unhatted $\Gamma$ just means ordinary sections, i.e., maps $\sigma:\cm\to\ce$ such that $p\circ\sigma=\id_\cm$.

It might seem strange at first that the functor of global sections is represented by the super vector space of sections of
$\ce\oplus\Pi\ce$. But the set of maps $\cm\to\ce$ only carries the structure of a vector space, not that of
a \emph{super} vector space. As is basically always the case, the set of maps between two super objects is itself
not super but can be enriched to become so. That is essentially due to the fact that the maps between super objects preserve parity.

As an example, the set of sections of the tangent bundle only consists of the even vector fields. To see the odd
ones as well we have add a parity changed copy of the tangent bundle. This is a large-scale version of the
simple fact that, for super vector spaces $V,W$, the inner Hom-object
\[
\ihom(V,W)\cong\Hom(V,W)\oplus\Hom(V,\Pi W)\cong\Hom(V,W\oplus\Pi W).
\]
This inner Hom object is the object which is usually of interest; the actual morphisms $V\to W$ only consitute its
even part.

\section{Supersmooth morphisms and their composition}

Following the general principles presented in section 2 the diffeomorphism supergroup $\sdiff(\cm)$ has to be
a subfunctor of the inner Hom-object $\ihom(\cm,\cm)$. The latter is defined as a functor $\catgr\to\catsets$ by setting
\[
\ihom(\cm,\cm')(\Lambda):=\Hom(\cp(\Lambda)\times\cm,\cm').
\]
and by the assignment of
\begin{eqnarray}
\label{eqn:scinftymap}
\ihom(\cm,\cm')(\varphi):\ihom(\cm,\cm')(\Lambda) &\to& \ihom(\cm,\cm')(\Lambda')\\
\nonumber
\sigma &\mapsto& \sigma\circ(\cp(\varphi)\times\id_\cm)
\end{eqnarray}
to each $\varphi:\Lambda\to\Lambda'$. We shall call the elements
of $\ihom(\cm,\cm')(\Lambda)$ supersmooth morphisms.
Note the similarity of this definition with that of the functor $\hat{\Gamma}(\cm,\ce)$ of sections of a super
vector bundle given in the last section: the higher points of the inner Hom object are morphisms of families over superpoints. For more motivation, see \cite{Sachse08A-Categorical-Formulation-of-Superalgebra-and-Supergeometry} and \cite{S:thesis}.

\subsection{Composition of morphisms and the unit element}

Let $\cm,\cm',\cm''$ be supermanifolds and fix $\Lambda\in\catgr$
for the moment. For two supersmooth maps
$f\in \Hom(\cp(\Lambda)\times\cm,\cm')$ and
$g\in\Hom(\cp(\Lambda)\times\cm',\cm'')$, the composition $g
\circ(\id_{\cp_{\Lambda}}\times
f)$ is in $\Hom(\cp(\Lambda)\times\cm,\cm'')$. This defines a map
\begin{multline*}
\icomp_{\Lambda}\from 
\ihom(\cm,\cm')(\Lambda)\times \ihom(\cm',\cm'')(\Lambda)\to
\ihom(\cm,\cm'')(\Lambda),\\
(f,g)\mapsto g \circ(	\id_{\cp_{\Lambda}}\times f).
\end{multline*}
If $\Lambda$ varies over all objects of $\catgr$, then this in fact
defines a natural transformation
$\icomp\from \ihom(\cm,\cm')\times\ihom(\cm',\cm'')\Rightarrow 
\ihom(\cm,\cm'')$.

\begin{lemma}
The functor
\begin{eqnarray*}
e_{\cm}\from \catgr\to\catsets,\quad
\Lambda \mapsto \{\Pi_{\cm}\from \cp(\Lambda)\times\cm\to\cm\}
\end{eqnarray*}
is a subfunctor of $\ihom(\cm,\cm)$, which defines the unit 
in $\catsets^\catgr$ for the composition $\icomp$. Moreover,
$\icomp$ is associative, giving $\ihom(\cm,\cm)$ the structure of
a semi-group in $\catsets^\catgr$.
\end{lemma}
\begin{proof}
This is clear from the definition.
\end{proof}

From the above it is obvious what the diffeomorphism supergroup of
a supermanifold should be. It should be comprised by subfunctors
of $\ihom(\cm,\cm)$ which are invertible with respect to $\icomp$.
Like the composition and all other operations invertibility has to be a ``point-wise'' notion.

\begin{dfn}
 For $f\in\Hom(\cp(\Lambda)\times\cm,\cm)$, an
 inverse is defined to be a morphism
 $f^{-1}\in\Hom(\cp(\Lambda)\times\cm,\cm)$  such that
 \begin{equation*}
  (\id_{\cp(\Lambda)}\times f)\circ f^{-1}=(\id_{\cp(\Lambda)}\times f^{-1})\circ f=\Pi_\cm.
 \end{equation*}
\end{dfn}
An inverse need not exist, but if it exists it is unique. If it exists,
we call $f$ invertible.

\subsection{Explicit description of $\boldsymbol{\ihom(\cm,\cm)}$}

Before turning to the diffeomorphism supergropup we derive some explicit parametrization results on the spaces $\ihom(\cm,\cm)(\Lambda)$.
We have 
\[
\Hom_{\catsman}(\cp(\Lambda)\times\cm,\cm)\cong\Hom_{\catsalg}(C^\infty(\cm),C^\infty(\cm)\otimes\Lambda).
\]
Therefore, any morphism $\phi:\cp(\Lambda)\times\cm\to\cm$ is given by an algebra homomorphism (which we also denote
$\phi$) of the form
\[
\phi(f)=\alpha_0(f)+\sum_{i}\tau_i\alpha_i(f)+\sum_{i<j}\tau_i\tau_j\alpha_{ij}(f)+\ldots,
\]
where the sums run over the odd generators $\tau_1,\ldots,\tau_n$ of $\Lambda$ and each $\alpha_I$ is a
linear map $C^\infty(\cm)\to C^\infty(\cm)$ of parity the length $|I|$ of its index.

The image of $\phi$ under $\ihom(\cm,\cm)(\epsilon_\Lambda)$ is the morphism $C^\infty(\cm)\rightarrow C^\infty(\cm)$
given by $\alpha_0$ because $\epsilon_\Lambda$ is the map which mods out all nilpotent elements from $\Lambda$.

Before we prove the general statement, let us investigate the case $\Lambda=\Lambda_1=\realR[\tau]$ in detail to
gain some intuition. That $\phi$ is a homomorphism means that
\begin{multline*}
\phi(fg)=\phi(f)\phi(g) = (\alpha_0(f)+\tau\alpha_1(f))(\alpha_0(g)+\tau\alpha_1(g))\\
= \alpha_0(f)\alpha_0(g)+\tau\left[\alpha_1(f)\alpha_0(g)+(-1)^{p(f)}\alpha_0(f)\alpha_1(g)\right].
\end{multline*}
This means that $\alpha_0$ is itself a homomorphism of superalgebras. We also see that $\alpha_1$ is a derivation over $\alpha_0$. That means the following. We can view the homomorphism $\alpha_0$ as endowing $C^\infty(\cm)$ with
an additional module structure over itself. Let us for clarity denote this
module structure as $C^\infty(\cm)^{\alpha_{0}}$. Then $\alpha_1$ is a derivation from $C^\infty(\cm)$ to
$C^\infty(\cm)^{\alpha_{0}}$.

It follows from the existence of universal derivations \cite{L:Algebra} that one may then write $\phi$ as
\[
\phi=(1+\tau X)\circ\alpha_0
\]
where $X$ is an odd vector field on $\cm$. The precise statement about universal derivations is that
\[
\mathrm{der}_R(A,M)\cong\Hom_A(\Omega,M)
\]
where $R$ is a commutative ring, $A$ is a commutative $R$-algebra and $M, \Omega$ are $A$-modules. One checks that this continues to hold for supercommutative rings and their modules. $\Omega$ is universal in the sense that every derivation $D:A\to M$
factors uniquely as $D=f\circ d$ where $d\from A\to \Omega$ is a derivation depending only on $A$ and $f\from \Omega\to M$ is $A$-linear. In our case $d$ is the de Rham differential, $\Omega$ are the 1-forms, $A$ is $C^\infty(\cm)$ and $M$ is $C^\infty(\cm)^{\alpha_{0}}$.
Now since $\Omega$ is in our case the dual space to the vector fields $\cx(\cm)$ we find that
\[
\Hom_A(\Omega,M)\cong\cx(\cm)\otimes_{C^\infty(\cm)}C^\infty(\cm)^{\alpha_{0}}.
\]
So derivations $D:C^\infty(\cm)\to C^\infty(\cm)^{\alpha_{0}}$ are still vector fields but with a different module
structure over the functions.

One checks that in the case $\Lambda=\Lambda_2$, $\phi$ takes on the form
\begin{eqnarray*}
\phi &=& \exp(\tau_1 X_1+\tau_2 X_2+\tau_1\tau_2 X_{12})\circ\alpha_0\\
&=& (1+\tau_1X_1+\tau_2X_2+\frac{1}{2}\tau_1\tau_2X_{12})\circ\alpha_0.
\end{eqnarray*}
The general picture will be very similar, with each $\alpha_I$ contributing an additional vector field of parity
$|I|$. 

So apart from $\alpha_0$, which describes a morphism of $\cm$ into itself, the higher terms depending on nilpotent
parameters of the base $\cp(\Lambda)$ act ``infinitesimally'', that is, by derivations. 
This is a ramification of the fact that odd dimensions behave infinitesimally,
familiar for example from the Taylor expansion of superfunctions into powers of their nilpotent part which is formally equivalent to extending a function onto an (odd) infintesimal neighbourhood.

Note that if $\alpha_0$ is invertible, as will be the case for diffeomorphisms, the induced map $d\alpha_0$ on
vector fields is an isomorphism and
\begin{equation}
\label{differential}
X\circ\alpha_0=\alpha_0\circ d\alpha_0(X)
\end{equation}
for every vector field $X$. So in this case we may choose whether we pre- or postcompose with $\alpha_0$.

For the proof of the general case, let us introduce the following notation. By $\mathfrak{S}(a_1\cdots a_n)$ we
denote the symmetrization of the product $a_1\cdots a_n$, i.e.,
\[
\mathfrak{S}(a_1\cdots a_n)=\frac{1}{n!}\sum_{\sigma\in P(n)} a_{\sigma(1)}\cdots a_{\sigma(n)},
\]
where $P(n)$ is the group of permutations of $n$ elements. The expression $I=I_1+\ldots+I_j$ will denote the decomposition of 
the ordered set $I$ into an ordered $j$-tuple of subsets $I_1,\ldots,I_j$, each carrying the 
ordering induced from $I$. For example, $\{1,2\}=I_1+I_2$ consists of the four partitions
\[
\{\{\},\{1,2\}\},\quad \{\{1\},\{2\}\},\quad \{\{2\},\{1\}\},\quad \{\{1,2\},\{\}\}.
\]
The notation $I=I_1\cup\ldots\cup I_j$, on the other hand, denotes the decomposition of the ordered set 
$I$ into an
\emph{unordered} $j$-tuple of disjoint ordered subsets. So, $\{1,2\}=I_1\cup I_2$ consists of two partitions:
\[
\{\{\},\{1,2\}\},\quad \{\{1\},\{2\}\}.
\]

The following lemma will be useful.

\begin{lemma}
\label{sym}
Let $A$ be an algebra, $f,g\in A$, and let $a_1,\ldots,a_n$ be derivations of $A$. Then
\begin{equation*}
\mathfrak{S}(a_1\circ\ldots\circ a_n)(fg)=\sum_{\{1,\ldots,n\}=K+L}
\mathfrak{S}(a_K)(f)\mathfrak{S}(a_L)(g),
\end{equation*}
where for $K=\{k_1,\ldots,k_j\}$, $a_K$ denotes the composition
\[
a_K=a_{k_1}\circ\ldots\circ a_{k_j}.
\]
\end{lemma}
\begin{proof}
By the Leibniz rule, it is clear that 
$\mathfrak{S}(a_1\circ\ldots\circ a_n)(fg)$ will take the form
\[
\mathfrak{S}(a_1\circ\ldots\circ a_n)(fg)=\sum_{\{1,\ldots,n\}=K+L}\frac{N(K,L)}{n!}a_K(f)a_L(g),
\]
with some integer $N(K,L)$ denoting the multiplicity the $K,L$-summand.
Since the symmetrized product on the left hand side contains all possible orderings of the operators
$a_i$, all possible partitions of $\{1,\ldots,n\}$ into two ordered subsets will really 
appear on the right hand side. The summand with given $K$ and $L$ occurs exactly $(|K|+|L|)!/(|K|!|L|!)$
times, as one checks as follows: starting from an ordered sequence $K$ of indices, there are
$(|K|+|L|)!/|K|!$ ways to insert $|L|$ elements at arbitrary positions into it.
But since the ordering of $L$ is also fixed, one has to divide by the number 
of permutations of $L$. So we have
\begin{eqnarray*}
\mathfrak{S}(a_1\circ\ldots\circ a_n)(fg) &=& \sum_{K,L\subseteq\{1,\ldots,n\}}%
\frac{(|K|+|L|)!}{|K|!|L|!n!}a_K(f)a_L(g)\\
&=& \sum_{\{1,\ldots,n\}=K+L}\mathfrak{S}(a_K)(f)\mathfrak{S}(a_L)(g)
\end{eqnarray*}
\end{proof}

\begin{thm}
\label{sdiff}
Let $\phi:\cp(\Lambda_n)\times\cm\to\cm$ be a $\Lambda_n$-point of $\ihom(\cm,\cm)$. Then
$\phi$ is uniquely determined by its underlying morphism $\phi_0:\cm\rightarrow\cm$, as well as 
$2^{n-1}$ odd
and $2^{n-1}-1$ even vector fields $X_I$ on $\cm$ such that
\begin{equation}
\label{toprove}
\phi=\exp(\sum_{I\subseteq\{1,\ldots,n\}}\tau_IX_I)\circ\phi_0,
\end{equation}
where the sum runs over all increasingly ordered nonempty subsets and $\tau_I$ is the product of
the corresponding $\tau_i$'s.
\end{thm}

\begin{proof}
Write
\begin{equation}
\label{topro2}
\phi=\sum_{I\subseteq\{1,\ldots,n\}}\tau_I\alpha_I,
\end{equation}
where we now sum over all (including the empty) increasingly ordered subsets and each
$\alpha_I$ is a linear map $C^\infty(\cm)\to C^\infty(\cm)$ of parity $|I|$.
The homomorphism property of $\phi$ implies that
\begin{equation}
\left(\sum_{K\subseteq\{1,\ldots,n\}}\tau_K\alpha_K(fg)\right)=%
\left(\sum_{I\subseteq\{1,\ldots,n\}}\tau_I\alpha_I(f)\right)\cdot%
\left(\sum_{J\subseteq\{1,\ldots,n\}}\tau_J\alpha_J(g)\right).
\label{hom}
\end{equation}
Identifying \eqref{toprove} with the sum \eqref{topro2} rephrases the claim of the theorem as
\begin{equation}
\label{topro3}
\tau_I\alpha_I=\sum_{j=1}^{|I|}\sum_{I=I_1\cup\ldots\cup I_j}%
\mathfrak{S}\left((\tau_{I_1}X_{I_1})\circ\ldots\circ(\tau_{I_j}X_{I_j})\right)\circ\alpha_0.
\end{equation}
The summation runs over all partitions of $I$ into unordered tuples of subsets, each subset carrying
the ordering induced from $I$ (cf.~the definition of the notation $I=I_1\cup\ldots\cup I_j$ above). This will be
proved by induction on $|I|$.

For indices $I$ of length $|I|=0,1$, the assertion holds as we have seen above.
Assume the statement has been proven for indices up to length $k$. Then let $I=\{i_1,\ldots,i_{k+1}\}$
be an index of length $k+1$.
We must assure that \eqref{hom} holds, which means we must find the general solution $\alpha_I$ for
\begin{eqnarray}
\nonumber
\tau_I\alpha_I(fg) &=& \alpha_0(f)\tau_I\alpha_I(g)+(-1)^{p(f)}\tau_I\alpha_I(f)\alpha_0(g)\\
\label{topro4}
&& \sum_{\substack{I=K+L\\ K,L\neq\emptyset}}\tau_K\alpha_K(f)\tau_L\alpha_L(g).
\end{eqnarray}
Since $|K|,|L|\leq k$, it follows that $\tau_K\alpha_K$ and $\tau_L\alpha_L$ 
must have the form \eqref{topro3}. Therefore 
the sum in \eqref{topro4} can be written as
\begin{multline*}
\sum_{\substack{I=K+L\\ K,L\neq\emptyset}}\left(\sum_{j=1}^{|K|}\sum_{K=K_1\cup\ldots\cup K_j}
\mathfrak{S}\left((\tau_{K_1}X_{K_1})\circ\ldots\circ(\tau_{K_j}X_{K_j})\right)(f)\circ\right.\\
\left.\sum_{l=1}^{|L|}\sum_{L=L_1\cup\ldots\cup L_l}
\mathfrak{S}\left((\tau_{L_1}X_{L_1})\circ\ldots\circ(\tau_{L_l}X_{L_l})\right)(g)\right)\circ\alpha_0.
\end{multline*}
By Lemma \ref{sym}, this equals
\begin{equation*}
\sum_{j=2}^{|I|}\sum_{I=I_1\cup\ldots\cup I_j}%
\mathfrak{S}\left((\tau_{I_1}X_{I_1})\circ\ldots\circ%
(\tau_{I_j}X_{I_j})\right)(fg)\circ\alpha_0.
\end{equation*}
The general solution to equation \eqref{topro4} therefore reads
\begin{eqnarray*}
\nonumber
\tau_I\alpha_I &=&  \tau_IX_I\circ\alpha_0+\sum_{j=2}^{|I|}\sum_{I=I_1\cup\ldots\cup I_j}%
\mathfrak{S}\left((\tau_{I_1}X_{I_1})\circ\ldots\circ(\tau_{I_j}X_{I_j})\right)\circ\alpha_0\\
&=& \sum_{j=1}^{|I|}\sum_{I=I_1\cup\ldots\cup I_j}%
\mathfrak{S}\left((\tau_{I_1}X_{I_1})\circ\ldots\circ(\tau_{I_j}X_{I_j})\right)\circ\alpha_0,
\end{eqnarray*}
where $X_I$ is a vector field of parity $|I|$ on $\cm$.
\end{proof}

As we have expected all topological features of $\ihom(\cm,\cm)$ are completely determined
by its underlying space $\Hom(\cm,\cm)$ while all higher points are vector bundles over the latter
space.

\section{Basic properties of the diffeomorphism supergroup}

We now turn to the diffeomorphism supergroup $\sdiff$ and it's
structural analysis. We will see that, exactly as in the previous
subsection, all analytical difficulties pertain to the group
underlying $\sdiff$.

\subsection{Group structure of $\boldsymbol{\sdiff(\cm)}$ in $\boldsymbol{\catsets^\catgr}$}
Define for each $\Lambda\in\catgr$ a set $\sdiff(\cm)(\Lambda)$ by setting
\begin{equation*}
\sdiff(\cm)(\Lambda)=\{f\in\ihom(\cm,\cm)(\Lambda)\mid f\textrm{ invertible}\}.
\end{equation*}
Clearly, each of these sets is a group. Therefore if we can show that they form a functor in
$\catsets^\catgr$, this functor will be a group object in $\catsets^\catgr$. 
In fact we will show that $\sdiff(\cm)$ is a subfunctor of $\ihom(\cm,\cm)$. 

\begin{prop}
\label{sdhom}
For each $\Lambda\in\catgr$ and each morphism $\varphi:\Lambda\to\Lambda'$, the restriction of 
$\ihom(\cm,\cm)(\varphi)$ to $\sdiff(\cm)(\Lambda)$ induces a group homomorphism
\begin{equation*}
\sdiff(\cm)(\varphi):\sdiff(\cm)(\Lambda)\to\sdiff(\cm)(\Lambda').
\end{equation*}
\end{prop}
\begin{proof}
Applying the definition \eqref{eqn:scinftymap} to the neutral element $\Pi_\cm:\cp(\Lambda)\times\cm\to\cm$,
we see immediately that
\begin{equation*}
\Pi_\cm\circ(\cp(\varphi)\times\id_\cm)=\Pi_\cm,
\end{equation*}
i.e., $\ihom(\cm,\cm)(\varphi)$ maps the unit element to the unit element. Now let 
$f,g\in\sdiff(\cm)(\Lambda)$ be given. We have to show that
\begin{equation*}
\ihom(\cm,\cm)(\varphi)(g\circ f)=(\ihom(\cm,\cm)(\varphi)(g))\circ(\ihom(\cm,\cm)(\varphi)(f)).
\end{equation*}
It is most insightful to compare the definition of the two functors. The left hand side corresponds to 
the composition
\begin{equation}
\label{f2}
\xymatrix@1{\cp(\Lambda')\times\cm \ar[rr]^{(\cp(\varphi),\id_\cm)} &&%
\cp(\Lambda)\times\cm \ar[rr]^{(\id_{\cp(\Lambda)},f)} && \cp(\Lambda)\times\cm \ar[r]^g &\cm},
\end{equation}
while the right hand side corresponds to
\begin{multline}
\label{f3}
\xymatrix@1{\cp(\Lambda')\times\cm \ar[rrr]^(.43){(\id_{\cp(\Lambda')},\cp(\varphi),\id_\cm)} &&&%
\cp(\Lambda')\times\cp(\Lambda)\times\cm \ar[rr]^(.7){(\id_{\cp(\Lambda')},f)} &&}\\
\xymatrix{&\ar[r]& \cp(\Lambda')\times\cm \ar[rr]^{(\cp(\varphi),\id_\cm)} &&%
\cp(\Lambda)\times\cm \ar[r]^g &\cm}.
\end{multline}
Let now $m\in\cm(\Lambda'')$ be some $\Lambda''$-point of $\cm$, $p\in\cp(\Lambda')(\Lambda'')$ be
a $\Lambda''$-point of $\cp(\Lambda')$ and let $q\in\cp(\Lambda)(\Lambda'')$ be its image under
$\cp(\varphi)$, i.e., $q=\cp(\varphi)(p)$. Then \eqref{f2} will map the pair $(p,m)$ to
\begin{equation*}
\xymatrix@1{(p,m) \ar@{|->}[r] & (q,m) \ar@{|->}[r] & (q,f_{\Lambda''}(q,m)) \ar@{|->}[r] &%
g(q,f_{\Lambda''}(q,m))}.
\end{equation*}
On the other hand, \eqref{f3} will map $(p,m)$ as
\begin{equation*}
\xymatrix@-=12pt{(p,m) \ar@{|->}[r] & (p,q,m) \ar@{|->}[r] & (p,f_{\Lambda''}(q,m)) \ar@{|->}[r] &%
(q,f_{\Lambda''}(q,m)) \ar@{|->}[r] & g(q,f_{\Lambda''}(q,m))}.
\end{equation*}
This shows that all components of the two functor morphisms \eqref{f2} and \eqref{f3} are indeed identical.
\end{proof}

\begin{cor}
$\sdiff(\cm)$ is a subfunctor of $\ihom(\cm,\cm)$ and a group object in $\catsets^\catgr$.
\end{cor}
\begin{proof}
By Proposition \ref{sdhom}, for $\varphi:\Lambda\to\Lambda'$, $\ihom(\cm,\cm)(\varphi)$ maps invertible 
morphisms to invertible morphisms, so the restriction of $\ihom(\cm,\cm)(\varphi)$ to
$\sdiff(\cm)(\Lambda)$ is well-defined. This means that the inclusion 
$\sdiff(\cm)\subset\ihom(\cm,\cm)(\varphi)$ is a functor morphism, and thus $\sdiff(\cm)$ is a
subfunctor. Since each $\sdiff(\cm)(\Lambda)$ is a group and each $\sdiff(\cm)(\varphi)$ is a group
homomorphism, the second assertion is clear.
\end{proof}

\subsection{Factoring out the underlying group $\boldsymbol{\Aut(\cm)}$}

Associated with the null object $\Lambda_0=\realR$ of $\catgr$ are the underlying points
(or $\realR$-points) of a supergroup. In our case,
the group $\sdiff(\cm)(\realR)$ obviously consists of the
invertible elements of $\End(\cm)$, i.e., of the automorphisms of $\cm$.
We shall denote this group by $\Aut(\cm)$. As mentioned above, the
initial and final morphisms $c_\Lambda\from \realR\to \Lambda$ and
$\epsilon_\Lambda\from \Lambda\to \realR$ of $\catgr$ furnish
canonical inclusions, resp. projections

\begin{eqnarray*}
\sdiff(c_\Lambda) &:& \Aut(\cm)\hookrightarrow\sdiff(\cm)(\Lambda),\\
\sdiff(\epsilon_\Lambda) &:& \sdiff(\cm)(\Lambda)\to\Aut(\cm).
\end{eqnarray*}

This turns $\Aut(\cm)$ into a subgroup of $\sdiff(\cm)(\Lambda)$ and
each $\cn(\cm)(\Lambda):=\ker(\sdiff(\cm)(\epsilon_{\Lambda}))$ is a
complementary normal subgroup. Since this construction is functorial,
the next lemma is immediate (compare also the discussion in Section \ref{sect:sgrp}).

\begin{lemma}\label{lem:semidirceProductDecomposition}
 The assignment $\Lambda\mapsto \cn(\cm)(\Lambda)$ defines a normal
 super subgroup $\cn(\cm)$ of $\sdiff(\cm)$ and we have
 \begin{equation*}
  \sdiff(\cm) \cong \cn(\cm)\rtimes \Aut(\cm)
 \end{equation*}
 (where we regard $\Aut(\cm)$ as the constant supergroup
 $\Lambda\mapsto \Aut(\cm)$).
\end{lemma}

As pointed out in Section \ref{sect:sgrp}, $\sdiff(\cm)$ even splits
as a direct sum in $\catsets^\catgr$. The splitting as a semidirect will, however,
even hold as Lie supergroups.
For each $f\from \cp(\Lambda)\times \cm\to\cm$ in $\sdiff(\Lambda)$,
the automorphism $\sdiff(\epsilon_{\Lambda})(f)$ of $\cm$
is given by
\begin{equation*}
 \cm\cong\cp(\realR)\times\cm
  \xrightarrow{\cp(\epsilon_{\Lambda})\times \id_{\cm}}
 \cp(\Lambda)\times\cm \xrightarrow{f}\cm
\end{equation*}
We call this the automorphism underlying $f$. That it is actually 
invertible is due to the fact that $\sdiff$ defines a functor.
From this it follows that $\cn(\cm)(\Lambda)$ consists of maps depending
non-trivially on the odd coordinates of $\cp_\Lambda$ and
whose underlying automorphism is the identity of $\Aut(\cm)$.

In Section \ref{sect:AutAsAFrechetLieGroup}, we shall put a supersmooth
structure on $\sdiff(\cm)$ with the aid of the decomposition from Lemma
\ref{lem:semidirceProductDecomposition}. This becomes feasible because
we shall derive charts for $\sdiff(\cm)$ respecting this decomposition.

\subsection{Invertibility of morphisms}

In this section we shall obtain an
explicit inversion formula for supersmooth diffeomorphisms.

\begin{thm}
\label{sdiffinv}
A supersmooth morphism $\varphi:\cp(\Lambda)\times\cm\to\cm$ is invertible if and only if its underlying
morphism $\varphi_\realR:\cm\to\cm$ is invertible. In this case, writing the algebra homomorphism $\phi$ as
\begin{equation*}
\phi=\exp(\sum_{I\subseteq\{1,\ldots,n\}}\tau_IX_I)\circ\phi_0
\end{equation*}
(in the notation of Theorem \ref{sdiff}), its inverse is given by
\begin{equation}
\label{invhom}
\phi^{-1}=\phi_0^{-1}\circ\exp(-\!\!\sum_{I\subseteq\{1,\ldots,n\}}\!\tau_IX_I).
\end{equation}
\end{thm}
\begin{proof}
We have to show that
\begin{equation}
\label{topro5}
\exp(-\!\!\sum_{I\subseteq\{1,\ldots,n\}}\!\tau_I X_I)\circ%
\exp(\!\sum_{J\subseteq\{1,\ldots,n\}}\!\tau_J X_J)=\id_{C^\infty_\cm(\cm)}.
\end{equation}
We can write
\begin{equation}
\label{topro6}
\exp(-\!\!\sum_{I\subseteq\{1,\ldots,n\}}\!\tau_I X_I)\circ\exp(\!\sum_{J\subseteq\{1,\ldots,n\}}\!\tau_J X_J)=%
1+\sum_K\tau_K\alpha_K
\end{equation}
by expanding both exponentials. Using \eqref{topro3}, we rewrite the expression on the left hand side as
\begin{multline*}
1+\left(\sum_{j=1}^{|I|}\sum_{I=I_1\cup\ldots\cup I_j}%
\mathfrak{S}\left((-\tau_{I_1}X_{I_1})\circ\ldots\circ(-\tau_{I_j}X_{I_j})\right)\right)\circ\\
\left(\sum_{k=1}^{|J|}\sum_{J=J_1\cup\ldots\cup J_k}%
\mathfrak{S}\left((\tau_{J_1}X_{J_1})\circ\ldots\circ(\tau_{J_k}X_{J_k})\right)\right)
\end{multline*}
Now $\tau_K\alpha_K$ on the right hand side of \eqref{topro6} is a sum over all partitions of $K$ 
into ordered tuples of
subsets. Pick one such tuple $\{K_1,\ldots,K_n\}$; the tuple, and each of the $K_i$, is ordered, and 
their union is $K$. On the left hand side, we have the corresponding sum
\begin{equation*}
\frac{1}{k!(n-k)!}\sum_{k=0}^n(-1)^k(\tau_{K_1}X_{K_1})\circ\ldots\circ(\tau_{K_n}X_{K_n})
\end{equation*}
of all ways of realizing this sequence of indices by contributions from either two of the exponentials in
\eqref{topro5}. But
\[
\sum_{k=0}^n\frac{1}{k!(n-k)!}(-1)^k=\frac{1}{n!}(1+(-1))^n=0.
\]
Therefore, each $\alpha_K$ on the right hand side of \eqref{topro6} receives only vanishing contributions,
and thus \eqref{topro5} holds.
\end{proof}

\begin{cor}
$\sdiff(\cm)$ is the restriction of $\ihom(\cm,\cm)$ onto
$\Aut(\cm)\subset\ihom(\cm,\cm)(\realR)$.
\end{cor}

%

\section{The Lie supergroup $\boldsymbol{\sdiff(\cm)}$}

The analytically involved part of the supersmooth structure on
$\sdiff(\cm)$ comes from the underlying group
$\Aut(\cm)$. In this section we show how to put a Lie group
structure on it. Up to a nilpotent semidirect factor,
$\Aut(\cm)$ can be identified with the automorphism group of a
finite-dimensional vector bundle, so that we can borrow the
smooth structure on it from
\cite{Wockel07Lie-group-structures-on-symmetry-groups%
-of-principal-bundles}. In order to do so, we have to assume that
the underlying manifold $\cm(\realR)$ is compact throughout.

\subsection{The structure sheaf of a supermanifold}
\label{sect:theStructureSheafOfASuperManifold}

The connection between finite-dimensional supermanifolds and vector bundles
is most easily described in the ringed-space picture (cf.\ 
\cite{DeligneMorgan99Notes-on-supersymmetry}), which we
will switch to for this and the following subsection. How to get from a super 
manifold in the categorical sense to the ringed space is described in
\cite[Sect.\ 5.1]{Sachse08A-Categorical-%
Formulation-of-Superalgebra-and-Supergeometry}.

When viewed as a ringed space, an $m|n$-dimensional supermanifold
$\cm$ is an $m$-dimensional manifold $M$, together with a sheaf
$C^{\infty}_{\cm}$ of $\Z_{2}$-graded supercommutative $\realR$ algebras
(i.e., $a\cdot b=(-1)^{|a|\cdot|b|}b\cdot a$ for homogeneous elements), which is locally
isomorphic to $C^{\infty}_{\realR^{m}}\otimes \Lambda_{n}$ where
$C^{\infty}_{\realR^{m}}$ is the sheaf of ordinary smooth functions on $\realR^m$. A
morphism between supermanifolds in this picture is a smooth morphism of the underlying
manifolds together with a morphism of sheaves.

Recall that the structure sheaf
$C^\infty_\cm$ of a supermanifold $\cm$ is filtered by the powers of the
nilpotent ideal sheaf $\cj\subset C^\infty_\cm$, i.e., \[
C^\infty_\cm\supset\cj\supset\cj^2\supset\ldots \] The sheaf is not
$\intZ$-graded, however, because morphisms of superalgebras only
preserve the $\intZ/2$-degree. Dividing out $\cj$ yields the underlying
manifold $M$,
and the quotient morphism $C^\infty_\cm\to C^\infty_\cm/\cj$ endows us
with a canonical embedding $\mathrm{cem}:M\hookrightarrow\cm$ as
a closed subsupermanifold. This construction is functorial, i.e., we
obtain a functor $\mathrm{red}:\catsman\to\catman$.

The sheaf $\cj/\cj^2$ has a natural $C^\infty_\cm/\cj$-module structure
on it, given by $[f]\cdot [\sigma]=[f\cdot \sigma]$. This turns
$\cj/\cj^2$ into a locally free sheaf of modules over $C^{\infty}_{\cm}$, which in
turn gives rise to a smooth vector bundle $E\to M$ with $\Gamma(E)\cong
\cj/\cj^2$. By Batchelor's Theorem
\cite{Batchelor79The-structure-of-supermanifolds} there exists a
(non-canonical) isomorphism $\xi\from \Gamma(\Lambda^{\bullet} E)\to
C^{\infty}_{\cm}$ covering $\id_{M}$, i.e., $\xi$ preserves the
$\Z_{2}$-grading. However, each two choices $\xi,\xi'$ give rise to an
isomorphism $\xi^{-1} \circ \xi'\from \Lambda^{\bullet} E \to
\Lambda^{\bullet} E$ covering $\id_{M}$, which gives in particular
rise to a vertical bundle automorphism $E\to E$. We shall call such a 
pair $(E,\xi)$ a \emph{vector bundle associated with} $\cm$.

\subsection{The super Lie algebra $\boldsymbol{\cx(\cm)}$}
\label{sect:TheSuperLieAlgebraOfSuperVectorFields}

The ringed space picture also provides a very accessible way to deal
with the Lie superalgebra $\cx(\cm)$ of vector fields. By Lemma
\ref{lem:superSectionsAreSuperRepresenatble}, the functor
$\wh{\Gamma}(\cm,\ct\cm)$ is superrepresentable and the ringed space
picture provides explicitly a $\Z_{2}$-graded vector space
representing $\cx(\cm)$ as in Example
\ref{ex:superrepresentableOlrModlue}.

The structure sheaf $C^{\infty}_{\cm}$ is a sheaf of super commutative
$\Z_{2}$-graded algebras on $M$. Thus it has a $\Z_{2}$-graded sheaf of
(even and odd) derivations, which we denote by $\Der(C^{\infty}_{\cm})$.
In local coordinates $x_i,\theta_j$, an even derivation
has the general form
\begin{equation}\label{eqn:evenDerivationInLocalCoordinates}
 X=\sum_{i=1}^m\sum_{\substack{I\subseteq\{1,\ldots,n\}\\
 |I|\,\,\mathrm{even}}}f_{iI}\theta_I\pder{}{x_i}+
 \sum_{j=1}^n\sum_{\substack{J\subseteq\{1,\ldots,n\}\\
 |J|\,\,\mathrm{odd}}} g_{jJ}\theta_J\pder{}{\theta_j}
\end{equation}
and an odd derivation has the general form
\begin{equation}\label{eqn:oddDerivationInLocalCoordinates}
 X=\sum_{i=1}^m\sum_{\substack{I\subseteq\{1,\ldots,n\}\\
 |I|\,\,\mathrm{odd}}}f_{iI}\theta_I\pder{}{x_i}+
 \sum_{j=1}^n\sum_{\substack{J\subseteq\{1,\ldots,n\}\\
 |J|\,\,\mathrm{even}}} g_{jJ}\theta_J\pder{}{\theta_j}
\end{equation}
where the sums run over all increasingly ordered subsets and $\theta_I$
denotes the product of the corresponding $\theta_j$'s in that same
order. The action of $X$ on
$f=\sum_{K\se \{1,\ldots,n\}}f_{K}\theta_{K}\quad\in C^{\infty}_{\cm}$ is then given by
\begin{eqnarray}\label{eqn:actionOfDerivationInLocalCoordinates1}
 \frac{\partial}{\partial x^{i}} f &:=& \sum_{K\se \{1,\ldots,n\}}
 \frac{\partial f_{K}}{\partial x^{i}} \theta_{K}\\
\label{eqn:actionOfDerivationInLocalCoordinates2}
 \frac{\partial}{\partial \theta_{j}} f &:=& \sum_{j\in K\se
 \{1,\ldots,n\}} f_{K}\theta_{K-\{j\}}\mathrm{sgn}(j,K)
\end{eqnarray}
where $\mathrm{sgn}(j,K)$ is the sign arising from moving $\frac{\partial}{\partial \theta_{j}}$
past the elements left of $\theta_j$ in $\theta_K$. 
The super commutator
$[X,Y]=XY-(-1)^{|X||Y|}YX$ turns $\Der(C^{\infty}_{\cm})$
into a Lie superalgebra.

From the above representation it also follows that 
$\Der(C^{\infty}_{\cm})$ can be endowed with a Fr\'echet topology, which
is induced by the embedding
\begin{equation*}
 \Der(C^{\infty}_{\cm})\hookrightarrow 
  \prod_{i\in I}\Der(\left.\cm\right|_{U_{i}})
  \cong \Der(C^{\infty}(U_{i})\otimes \Lambda_{n})
\end{equation*}
(for $(U_{i})_{i\in I}$ an open covering of $M$ with
$\left.\cm\right|_{U_{i}}\cong C^{\infty}(V_{i})\otimes \Lambda_{n}$
and $V_{i}\se \R^{n}$ open) and endowing 
$\Der(C^{\infty}(U_{i})\otimes \Lambda_{n})$ with the natural Fr\'echet
topology. Since the natural operations are continuous with respect to
this topology, this turns $\Der(\cm)$ into a Fr\'echet super
Lie algebra. 

From the local representation of a derivation in
\eqref{eqn:evenDerivationInLocalCoordinates} and
\eqref{eqn:oddDerivationInLocalCoordinates} and Lemma
\ref{lem:superSectionsAreSuperRepresenatble} it also follows that
$\ol{\Der(C^{\infty}_{\cm})}\cong \cx(\cm)$ as $\olr$-modules,
which enriches $\cx(\cm)$ to a Fr\'echet super Lie algebra.

\subsection{The structure of $\boldsymbol{\cx(\cm)}$ and
$\boldsymbol{\Aut(\cm)}$}

An automorphism of $\cm$ is a homomorphism of its structure sheaf,
i.e., it preserves the grading.
The Lie algebra of $\Aut(\cm)$ is therefore the algebra of grading-preserving, i.e., even, vector fields
$\cx(\cm)_{\bar{0}}$.

In view of the action of vector fields on functions on $\cm$ described in
\eqref{eqn:actionOfDerivationInLocalCoordinates1}
and \eqref{eqn:actionOfDerivationInLocalCoordinates2} we readily identify the
even vector fields whose action induces the identity on the underlying manifold:
these are the ones which do not contain a summand $f_i(x)\frac{\partial}{\partial x_i}$.
That is, in their local representation each coefficient function is at least of degree one
in the odd variables.
Similarly, if an even vector field $X$ only has coefficient functions of degree $\geq 2$ in the odd
variables it will induce the identity on $C^\infty_\cm/\cj^2$ and thus on the underlying manifold 
as well as on any vector bundle describing $\cj/\cj^2$ and so on.

We can define a filtration on $\cx(\cm)$ analogous to that on $C^\infty_\cm$ by giving each odd coordinate
(in some arbitrary local coordinate system) degree $1$ and each derivative $\frac{\partial}{\partial\theta_j}$
degree $-1$. Then we define
$\cx(\cm)^{(k)}$ as the ideal in $\cx(\cm)$ consisting of even vector fields whose
local coordinate representations are of degree at least $k$ (in the 
odd variables). This defines a filtration which is independent of the choice of local coordinates:
the exact number of odd variables in a superfunction is not preserved under coordinate
changes, but it never decreases, which is precisely the statement that coordinate changes
respect the filtration of $C^\infty_\cm$ by powers of the nilpotent ideal $\cj$.

In particular, $\evf$ consists of all $\cx(\cm)^{(k)}$ with even $k$, the odd vector fields have
odd degrees.
So, for example, $\evf^{(0)}/\evf^{(2)}$ locally consists of linear combinations of vector fields of the form
$f(x)\partial_{x_j}$ and $g(x)\theta_i\partial_{\theta_j}$ and therefore acts nontrivially on the
underlying manifold $M$ as well as on the associated vector bundles of $\cm$.

The subgroup of $\Aut(\cm)$ which induces the identity on $C^\infty_\cm/\cj^k$ will be denoted as
$\op{Nil}_{\cm}^{(k)}$. If we write sloppily $\cm/\cj^k$ for the ringed space obtained by
dividing out the $k$-th power of $\cj$ then we could also define $\op{Nil}_{\cm}^{(k)}$ as the
kernel of the natural map $\Aut(\cm)\to\Aut(\cm/\cj^k)$.


\begin{prop}\label{prop:exponentialFunctionOnNilpotentPart}
 If $k\geq 2$, then $\evf^{(k)}$ consists of nilpotent derivations, $\op{Nil}_{\cm}^{(k)}$ 
 of unipotent automorphisms (of $C^\infty_\cm$ respectively), and the exponential map
 \[
  \exp:\evf^{(k)}\to\mathrm{Nil}_{\cm}^{(k)}
 \]
 is bijective.
\end{prop}

\begin{proof}
 An element $X$ of $\evf^{(k)}$ is an even derivation of $C^{\infty}_{\cm}$
 such that $|I|$ and $|J|$ in its coordinate representation
 \eqref{eqn:evenDerivationInLocalCoordinates} are bounded below by $k$ and $k+1$, respectively.
 With the definition of the action of $\evf$ on
 $C^{\infty}_{\cm}$ in \eqref{eqn:actionOfDerivationInLocalCoordinates1}
 and \eqref{eqn:actionOfDerivationInLocalCoordinates2}
 one sees that applying $X$ to $f$ raises the length of the indices of
 the odd variables $\theta_{K}$ by at least $k$. From this it follows that $X$
 acts nilpotently if $k\geq 2$.

 Thus $\exp(X)$ actually is a finite sum and the exponential map is
 well-defined. Moreover, $\exp(X)\in \op{Nil}_{\cm}^{(k)}$, since in any local coordinate
 system,
 $\cj^{k}$ is generated by $\{\theta_{K}:|K|\leq k\}$ over $C^\infty/\cj$ and
 thus $\exp(X)$ acts trivially on $\cm/\cj^{k}$. By the same argument
 as above, an element $\varphi\in\op{Nil}_{\cm}^{(k)}$ is unipotent if
 $k\geq 2$. Moreover,
 \begin{equation*}
  \log(\varphi):=\sum_{l\geq 1} (-1)^{l}\frac{(\varphi -\id)^{l}}{l}
 \end{equation*}
 defines an inverse map for $\exp$ (cf.\
 \cite{Essen00Polynomial-automorphisms-and-the-Jacobian-conjecture}).
\end{proof}


The group $\op{Nil}_{\cm}:=\op{Nil}(\cm)^{(2)}$ will be particularly
important, for it can be turned into a semidirect factor in $\Aut(\cm)$, albeit non-canonically. The
corresponding quotient is $G:=\Aut(\cm/\cj^{2})$, which we can embed
into $\Aut(\cm)$ by choosing a vector bundle $E\to M=\cm(\realR)$
associated with $\cm$. In fact, $\cj^{2}=\Lambda^{\geq 2}E$
in the case that $C^\infty_\cm=\Lambda_{C^\infty_M}^{\bullet}E$ and thus automorphisms of
$\cm/\cj^{2}=E$ become the same as vector bundle automorphisms
of $E$. On the other hand, each $f\in \Aut(E)$ acts as an automorphism
on the sheaf of sections of $E$ and this determines uniquely an
automorphism of $\Lambda^{\bullet}E$.

\begin{cor}\label{cor:splitOfAutomorphismGroup}
 The group $\mathrm{Nil}_\cm$ fits into an exact sequence
 \begin{equation}
  \label{eseq}
  1\longrightarrow\mathrm{Nil}_\cm\longrightarrow\Aut(\cm)\longrightarrow G\longrightarrow 1.
 \end{equation}
 This sequence splits (non-naturally in $\cm$) by a morphism
 $\sigma_{E}\from G\rightarrow \Aut(\cm)$, which depends on a choice of a vector
 bundle $E$ associated with $\cm$ and we have
 $\Aut(\cm)\cong \op{Nil}_{\cm}\rtimes_{E} G$.
\end{cor}


%
%

\subsection{$\boldsymbol{\Aut(\cm)}$ as a Fr\'echet--Lie group}\label{sect:AutAsAFrechetLieGroup}

Just as the diffeomorphism group of a compact manifold is modeled on the
Lie algebra of smooth vector fields, we will model $\sdiff(\cm)$ on the
superrepresentable $\olr$-module $\cx(\cm)$ of super vector fields on
$\cm$. Consequently, we are seeking for a Lie group structure on
$\Aut(\cm)=\sdiff(\cm)(\Lambda_{0})$, which is modeled on
$\aut(\cm):=\wh{\Gamma}(\cx(\cm))(\Lambda_{0})=\evf$ (cf.\ Section
\ref{sect:superVectorBundles}). Pulling back a chart for this Lie group
structure along the terminal morphism 
$\epsilon_{\Lambda}\from\Lambda\to \Lambda_{0}$ then provides us with
charts for a Lie group structure on each $\sdiff(\Lambda)$. Since this
construction is functorial we will end up with a super Lie group structure
on $\sdiff$.

For the following construction we choose a vector bundle $(E,\xi)$
associated with $\cm$ as in Section
\ref{sect:theStructureSheafOfASuperManifold} and note that for a
different choice $(E,\xi')$ we have $\xi=\gamma\circ \xi'$ for an
automorphism $\gamma\from \Lambda^{\bullet} E \to
\Lambda^{\bullet} E$. We shall use $\gamma$ later on to show that
the smooth structure on $\Aut(\cm)$ does not depend on the choice of
$\xi$. We use $\xi$ to identify $\Aut(\cm)$ with $\Aut(\Lambda^{\bullet}
E)$, where the latter group denotes fiberwise algebra automorphisms
preserving the $\Z_{2}$-grading. Then Corollary
\ref{cor:splitOfAutomorphismGroup} yields the semidirect decomposition
\begin{equation*} 
\Aut(\Lambda^{\bullet }E)\cong \Aut(\Lambda^{\geq
2}E)\rtimes \Aut(E) 
\end{equation*} with respect to the natural action
of $\Aut(E)$ on $\Lambda^{\geq 2}E$. Now Proposition
\ref{prop:exponentialFunctionOnNilpotentPart} yields a bijective
exponential function 
\begin{equation*} 
\exp\from\aut(\Lambda^{\geq
2}E)\to\Aut(\Lambda^{\geq 2}E), 
\end{equation*} 
where
$\aut(\Lambda^{\geq 2}E)$ denotes the even derivations of
$\Lambda^{\geq 2}E$. We have seen in Section
\ref{sect:TheSuperLieAlgebraOfSuperVectorFields} how to put on
$\aut(\Lambda^{\geq 2}E)$ the structure of a Fr\'echet algebra and the
induced smooth structure on $\Aut(\Lambda^{\geq 2}E)$ turns it into a
Fr\'echet-Lie group. It thus remains to put a smooth structure on
$\Aut(E)$ and to show that the induced action is smooth.

\begin{thm}
 If $E\to M$ is a finite-dimensional vector bundle over the compact
 manifold $M$, then $\Aut(E)$ can be given the structure of a
 Fr\'echet--Lie group, modelled on the Fr\'echet space
 \begin{equation*}
  \gau(E)\oplus \cV(M),
 \end{equation*}
 where $\gau(E)$ denotes the Lie algebra of sections in the endomorphism
 bundle $\op{end}(E)$ and $\cV(M)$ the Lie algebra of vector fields on
 $M$, both endowed with the natural $C^{\infty}$-topology.
\end{thm}

\begin{proof}
 Since $E$ is finite-dimensional its frame bundle $F_{E}$ is so.
 The latter is a principal $\GL(V)$-bundle, where $V$ denotes
 the typical fiber of $E$ and the construction from
 \cite{Wockel07Lie-group-structures-on-symmetry-groups-of-%
 principal-bundles} yields a smooth structure on $\Aut(F_{E})$,
 modeled on $\gau(E)\oplus \cV(M)$. Using the canonical
 isomorphism $\Aut(F_{E})\cong \Aut(E)$ then induces a smooth structure
 on $\Aut(E)$.
\end{proof}

Note that the Lie algebra $\aut(E)$ of $\Aut(E)$ is only isomorphic to
$\gau(E)\oplus \cV(M)$ as a vector space but not as a Lie algebra.
In general, one only has an extension
\begin{equation*}
 0\to\gau(E)\to\aut(E)\to\cV(M)\to 0
\end{equation*}
of Fr\'echet--Lie algebras, which does \emph{not} split.
Moreover, charts for the smooth structure are not very handsome for in general they
cannot come from an exponential function. However, restricting to the
normal subalgebra $\gau(E)\trianglelefteq \aut(E)$ of sections in the endomorphism bundle, we have an exponential funtion
\begin{equation*}
 \exp\from \gau(E)\to \Gau(E),
\end{equation*}
where $\Gau(E)$ denotes the group of vertical bundle automorphisms of
$E$. This exponential function is given by taking the exponential 
function $\mathrm{End}(V)\to \GL(V)$ in each fiber and may be used to obtain 
a chart for the normal subgroup $\Gau(E)$ (cf.\ \cite[Th.\
1.11]{Wockel07Lie-group-structures-on-symmetry-groups-of%
-principal-bundles}). The inconvenience in the construction of a chart
on $\Aut(E)$ now comes from extending the chart
on $\Gau(E)$ to $\Aut(E)$, which mainly involves the construction of a
chart of $\Diff(M)$ on $\cV(M)$ (cf.\ \cite[Sect.\
2]{Wockel07Lie-group-structures-on-symmetry%
-groups-of-principal-bundles}).

\begin{cor}
 If $\cm$ is a finite-dimensional supermanifold such that the underlying
 manifold $M$ is compact, then $\Aut(\cm)$ carries the structure of
 a Fr\'echet--Lie group. If $(E,\xi)$ is a vector bundle associated
 with $\cm$, then $\Aut(\cm)$ is modeled on
 \begin{equation*}
 \aut(\Lambda^{\geq 2}E)\oplus\gau(E)\oplus \cV(M).
 \end{equation*} 
\end{cor}

\begin{proof}
 The preceding theorem yields a smooth structure on $\Aut(E)$ and the
 bijective exponential function
 $\exp \from \aut(\Lambda^{\geq 2}E)\to \Aut(\Lambda^{\geq 2}E)$ induces
 a smooth structure on $\Aut(\Lambda^{\geq 2}E)$. The induced action of
 $\Aut(E)$ on $\Aut(\Lambda^{\geq 2}E)$ is smooth, because
 the actions of $\Gau(E)$ on $\gau(E)$ and of
 $\Diff(M)$ on $C^{\infty}(M)$ are smooth, and on a unit
 neighborhood the $\Aut(E)$-action is given (in local coordinates)
 in terms of the $\Gau(E)$ and $\Diff(M)$-actions.
 From this it follows that
 \begin{equation*}
  \Aut(\Lambda^{\bullet}E)\cong \Aut(\Lambda^{\geq 2}E)\rtimes \Aut(E)
 \end{equation*}
 carries a Lie group structure, modeled on
 $\aut(\Lambda^{\geq 2}E)\oplus\gau(E)\oplus \cV(M)$. Now 
 $\xi\from \ul{\Lambda^{\bullet}E}\to \cm$ induces an isomorphism 
 $\Aut(\cm)\to \Aut(\ul{\Lambda^{\bullet}E})\cong  \Aut(\Lambda^{\bullet E})$. 
 Since two different $\xi$ differ by an
 equivalence of $\ul{\Lambda^{\bullet}}E$
 the smooth structure does not depend on this choice if we use $\xi$
 to transport this structure from $\Aut(\Lambda^{\bullet}E)$ to
 $\Aut(\cm)$.
\end{proof}

\subsection{Charts on $\boldsymbol{\sdiff(\cm)}$}

Denote by $\cx(\cm)$ the superrepresentable $\olr$-module of sections of
the tangent bundle of $\cm$. As we have seen, $\cx(\cm)$ is nothing else than the $\olr$-module
associated with the super vector space of vector fields on $\cm$. To equip $\sdiff(\cm)$ with a supersmooth
Lie group structure, modeled on $\cx(\cm)$, we start with an open zero
neighborhood $U\se \aut(\cm)$ and a chart $\Phi\from V\to
U$ for some open unit neighborhood $V$ of $\Aut(\cm)$. This defines an 
open subfunctor
\begin{equation*}
 \ol{U}\from \catgr\to\cattop,\quad 
 \left\{\begin{array}{ll}
  \Lambda\mapsto\cx(\cm)(\epsilon_{\Lambda})^{-1}(U)&
   \!\!\tx{on objects}\\
  \varphi\mapsto\left.
    \cx(\cm)(\varphi)\right|_{\cx(\cm)(\epsilon_{\Lambda})^{-1}(U)}&
   \!\!\tx{on morphisms}
 \end{array}\right.
\end{equation*}
of $\cx(\cm)$(note that each open subfunctor is of this kind, cf.\
\cite[Prop.\
4.8]{Sachse08A-Categorical-Formulation-of-Superalgebra-and%
-Supergeometry}). Likewise, we obtain a subfunctor
\begin{equation}\label{eqn:openSubfunctorOfSdiff}
 \ol{V}\from \catgr\to\catsets,\quad 
 \left\{\begin{array}{ll}
  \Lambda\mapsto\sdiff(\cm)(\epsilon_{\Lambda})^{-1}(V)&
   \!\!\tx{on objects}\\
  \varphi\mapsto\left.
    \sdiff(\cm)(\varphi)\right|_{\sdiff(\cm)(
     \epsilon_{\Lambda})^{-1}(V)}&
   \!\!\tx{on morphisms}
 \end{array}\right.
\end{equation}
of $\sdiff(\cm)$. We now wish to set up a Lie group structure on each
$\sdiff(\cm)(\Lambda)$ such that $\ol{V}$ becomes an open subfunctor
and such that we have a functorial isomorphism
$\Phi_{\Lambda}\from \ol{V}(\Lambda)\to \ol{U}(\Lambda)$ such that each $\Phi_{\Lambda}$
is a chart for the Lie group structure on $\sdiff(\cm)(\Lambda)$.
This then yields a super Lie group structure on $\sdiff$.

As in Lemma
\ref{lem:semidirceProductDecomposition}, the initial and final morphisms
in $\catgr$ furnish $\cx(\cm)$ with a functorial decomposition
$\cx(\cm)(\Lambda)\cong \mf{n}(\cm)(\Lambda)\rtimes \aut(\cm)$, where
$\mf{n}(\cm)(\Lambda):=\ker(\cx(\cm)(\epsilon_{\Lambda}))$ and
$\aut(\cm)=\evf$ is the Lie algebra of $\Aut(\cm)$.
Since $\mf{n}(\cm)(\Lambda)$ is the subspace of $\cx(\cm)(\Lambda)$ consisting of elements proportional
to (products of) odd generators of $\Lambda$, Proposition \ref{sdiff}
yields a bijective exponential function
\begin{equation*}
 \exp_{\Lambda}\from \mf{n}(\cm)(\Lambda)\to \cn(\cm)(\Lambda),\quad
\end{equation*}
which we use to endow each $\cn(\cm)(\Lambda)$ with a smooth
structure. As in Section \ref{sect:AutAsAFrechetLieGroup} one observes
that the $\Aut(\cm)$-action on $\evf$ and $\ovf$ is smooth
and thus that the action of $\Aut(\cm)$ on $\mf{n}(\Lambda)$ is smooth.
Thus $\Aut(\cm)$ also acts smoothly on $\cn(\cm)(\Lambda)$
and therefore each $\cn(\cm)(\Lambda)\rtimes \Aut(\cm)$ becomes
an infinite-dimensional Lie group, modeled on $\cx(\cm)(\Lambda)$. A
chart for this Lie group structure is given by
\begin{equation*}
 \log_{\Lambda}\times \Phi\from \cn(\cm)(\Lambda)\times V\to 
 \mf{n}(\cm)(\Lambda)\times U,
\end{equation*}
where $\log_{\Lambda}$ denotes the inverse map to $\exp_{\Lambda}$.

\begin{prop}
\label{prop:manfr}
 Endowing each $\sdiff(\cm)(\Lambda)$ with the topology just described
 turns $\sdiff(\cm)$ into a functor $\catgr\to \catman_{\op{Fr}}$, where
 $\catman_{\op{Fr}}$ denotes the category of Fr\'echet manifolds.	
\end{prop}

\begin{proof}
 We only have to verify that $\sdiff(\cm)(\phi)$ becomes a smooth
 morphism for each $\phi\from \Lambda\to \Lambda'$. From the
 construction of $\sdiff$ it follows that its restriction to
 $\cn(\cm)(\Lambda)\times V$ is given by
 \begin{equation*}
  \cn(\cm)(\varphi)\times \id_{V}.
 \end{equation*}
 It thus suffices to verify that the restriction to $\cn(\cm)(\Lambda)$
 is smooth, whose coordinate representation is $\mf{n}(\cm)(\varphi)$.
 Since the latter map is linear and continuous it is in particular
 smooth.
\end{proof}

\begin{prop}
\label{prop:schart}
 For each chart $\Phi\from V\to U$ of $\Aut(\cm)$ the functor $\ol{V}$
 as defined in \eqref{eqn:openSubfunctorOfSdiff} is an open subfunctor
 (with respect to the smooth structure just described).
 Moreover, the assignment $\Lambda \mapsto \log_{\Lambda}\times \Phi$
 constitutes a natural isomorphism $\ol{V}\to \ol{U}$ of functors
 $\catgr\to\catman_{\op{Fr}}$.
\end{prop}

\begin{proof}
 Since $\ol{U}(\Lambda)=U\times\mathfrak{n}(\cm)(\Lambda)$ for all $\Lambda$, $\ol{U}$ 
 is an open subfunctor of $\cx(\cm)$. On the other hand we have given $\ol{V}$ the
 topology pulled back from $\ol{U}$ via the bijection $\log_{\Lambda}\times\Phi$
 where $\Phi:V\to U$ is the underlying chart on $\Aut(\cm)$ so $\ol{V}$ is open.
 
 The very same argument applies to the smooth structure: we have endowed $\ol{V}$ with
 the smooth structure pulled back from $\ol{U}$, turning $\log_{\Lambda}\times\Phi$ and
 $\exp_\Lambda\times\Phi^{-1}$ into mutually inverse diffeomorphisms.
\end{proof}

Eventually, we may conclude the following theorem.

\begin{thm}
 $\sdiff(\cm)$ is a Fr\'echet super Lie group, modeled on the
 superrepresentable $\olr$-module of (even and odd) vector fields
 $\cx(\cm)$.
\end{thm}

\begin{proof}
Proposition \ref{prop:manfr} turns $\sdiff(\cm)$ into a functor $\catgr\to\catman_{\op{Fr}}$ and Proposition \ref{prop:schart} extends the chart $V$ around the identity on $\Aut(\cm)$ to a superchart on $\sdiff(\cm)$. 
This superchart $\olv$ can be translated to a superchart around any $\phi\in\Aut(\cm)$: since $\Aut(\cm)$ canonically embeds into each of the groups $\sdiff(\cm)(\Lambda)$, $\phi$ acts on each of the points $\olv(\Lambda)$ by left and right translation.

It remains to be shown that the transition functions between the charts obtained in this way are supersmooth.
The components are clearly smooth, so we just have to check the $\Lambda_{\bar{0}}$-linearity of the differential.

It is sufficient to study the intersection of $\olv$ and a chart $R_{\psi_0}\olv$ obtained from it by, say, right translation with an element $\psi_0\in\Aut(\cm)$. Then we see from \eqref{sdiff_mult} that every element in $R_{\psi_0}\olv(\Lambda_n)$ is of the form
\[
\exp(\sum_{I\subseteq\{1,\ldots,n\}}\tau_IX_I)\circ\phi_0\circ\psi_0.
\]
If such an element lies in $\olv(\Lambda_n)$ as well then the transition function will only affect the underlying part by identifying $\phi_0\circ\psi_0$ with some other $\phi_0'$ in $\olv(\realR)$. On the
nilpotent part $\cn(\cm)(\Lambda_n)$ in $\olv(\Lambda_n)\cong V\times\cn(\cm)(\Lambda_n)$ (that is an isomorphism in $\catman^\catgr$) the transition function acts as the identity. Its differential is thus the identity as well and therefore in particular $\Lambda_{\bar{0}}$-linear.

Had we instead used left translation to produce a superchart $L_{\psi_0}\olv$ then we would have found
\[
\psi_0\circ\exp(\sum_{I\subseteq\{1,\ldots,n\}}\tau_IX_I)\circ\phi_0=%
\exp(\sum_{I\subseteq\{1,\ldots,n\}}\tau_I\,d\psi_0^{-1}(X_I))\circ\psi_0\circ\phi_0.
\]
So in this case the transition function acts as $d\psi_0^{-1}$ on $\mathfrak{n}(\Lambda_n)\subset\cx(\cm)(\Lambda_n)$. Since $d\phi_0$ and its inverse are by definition extended to $\cx(\cm)(\Lambda)$ as $\Lambda_0$-linear maps the differential of the transition map is again $\Lambda_{\bar{0}}$-linear.
\end{proof}

\subsection{Supersmoothness of the Group multiplication}

Until now we have only turned $\sdiff(\cm)$ into a super manifold, but we actually want to turn it into a Lie supergorup. For this we have to show that the multiplication functor actually is supersmooth.
Given a function $f\in C^{\infty}(\cm)$, an automorphism $\phi_0\in \Aut(\cm)\cong\Aut(C^{\infty}(\cm))$ and a vector field $X\in \cx(\cm)\cong\Der(C^{\infty}(\cm))$ we have
\[
X \circ\phi_0(f)=(\phi_{0}\circ \phi_{0}^{-1}\circ X \circ \phi_{0})(X)=\phi_{0}\circ (d\phi_{0}^{-1}(X))(f),
\]
(cf. \eqref{differential}). For $X$ an even derivation, $(\phi_{0},X)\mapsto d \phi_{0}(X)$ is the adjoint action of $\Aut(\cm)$ its 
Lie algebra $\cx(\cm)_{\ol{0}}$ and thus the action of $\Aut(\cm)$ on $\cx(\cm)$ is smooth.

Now Theorem \ref{sdiff} permits us to derive an explicit formula for the group multiplication in coordinates. Let 
\begin{eqnarray*}
\phi &=& \exp(\sum_{I\subseteq\{1,\ldots,n\}}\tau_IX_I)\circ\phi_0,\\
\psi &=& \exp(\sum_{J\subseteq\{1,\ldots,n\}}\tau_JY_J)\circ\psi_0
\end{eqnarray*}
be two $\Lambda_n$-points of $\sdiff(\cm)$. Then we have
\begin{eqnarray}
\phi\circ\psi &=& \exp(\sum_{I\subseteq\{1,\ldots,n\}}\tau_IX_I)\circ\phi_0\circ%
\exp(\sum_{J\subseteq\{1,\ldots,n\}}\tau_JY_J)\circ\psi_0\label{sdiff_mult}
\\
\nonumber&=& \exp(\sum_{I\subseteq\{1,\ldots,n\}}\tau_IX_I)\circ%
\exp(\sum_{J\subseteq\{1,\ldots,n\}}\tau_J\,d\phi_0^{-1}(Y_J))\circ\phi_0\circ\psi_0.
\end{eqnarray}
For $\psi_0=\phi_0^{-1}$ and $X_I=-Y_I$ one recovers the inversion formula \eqref{invhom} for superdiffeomorphisms.

In the next section we will show that $\Aut(\cm)$ can be turned into a Fr\'echet Lie group acting smoothly on
vector fields. Assuming this we can show
\begin{prop}
The group multiplication in $\sdiff(\cm)$ is supersmooth.
\end{prop}
\begin{proof}
From the above formula it is evident that the multiplication morphism is smooth in every $\Lambda_{n}$-point, so it remains to check that it is also supersmooth.

To see that the differential of the multiplication is $\Lambda_{\bar{0}}$-linear it is sufficient to check
that the differentials of left and right translation are $\Lambda_{\bar{0}}$-linear. To check that it is in
turn enough to see that the action of a superdiffeomorphism $\phi\in\sdiff(\cm)(\Lambda)$ on $\cx(\cm)(\Lambda)$ is a $\Lambda_{\bar{0}}$-linear map.

This is shown in \cite{S:thesis}. More precisely it is shown that $\phi$ acts on a super vector field $Y$
by its differential
\begin{eqnarray*}
d\phi(Y) &=& \exp(-\!\sum_{I\subseteq\{1,\ldots,n\}}\tau_I L_{X_I})\circ d\phi_0(Y)
\end{eqnarray*}
where $L_X$ denotes the Lie derivative, i.e., the commutator of vector fields in this case. This action is
extended to all of $\cx(\cm)(\Lambda)$ in the usual way (i.e., by means of the functor $\bar{\cdot}$, cf. \eqref{fbar}). This means the action of $\phi$ on $\cx(\cm)(\Lambda)$ consists of a composition of $d\phi_0$ and brackets and is therefore by construction  $\Lambda_{\bar{0}}$-linear.
\end{proof}


\end{document}